\documentclass[leqno,11pt]{amsart}
\usepackage{amsmath}
\usepackage{amsfonts}
\usepackage{amssymb}
\usepackage{amsthm}

\usepackage{mathrsfs}
\usepackage{dsfont}
\usepackage{txfonts}
\usepackage{mathabx}

\usepackage{graphicx}
\usepackage[all]{xy}
\usepackage{pgf,tikz}
\usetikzlibrary{arrows}
\usepackage{subfigure}

\usepackage[pdftex]{hyperref}

\usepackage{color}
\usepackage[shortlabels]{enumitem}

\newtheorem{theorem}{Theorem}[section]

\newtheorem{proposition}[theorem]{Proposition}

\theoremstyle{definition}
\newtheorem{definition}[theorem]{Definition}
\newtheorem{notation}[theorem]{Notation}

\theoremstyle{remark}
\newtheorem{remark}[theorem]{Remark}

\newcommand{\bd}{\partial}
\newcommand{\R}{\mathbb{R}}
\newcommand{\rn}{\mathbb{R}^n}
\newcommand{\rno}{\mathbb{R}^{n+1}}

\newcommand{\cL}{\mathcal{L}}

\newcommand{\cv}{\text{Conv}}

\newcommand{\cA}{\Lambda}

\newcommand{\cW}{\theta}

\newcommand{\Sj}{\Gamma}
\newcommand{\bpsi}{\widebar{\psi}}
\newcommand{\grad}{\upsilon}

\newcommand{\subscript}[2]{$#1 _ #2$}

\title{The evolution of complete non-compact graphs \\ by powers of  Gauss  curvature }

\author{Kyeongsu Choi}
\address{ {\bf Kyeongsu Choi:} Department of Mathematics, Columbia University, 2990 Broadway, New York, NY 10027, USA.}
\email{kschoi@math.columbia.edu}
\author{Panagiota Daskalopoulos}
\address{ {\bf P. Daskalopoulos:} Department of Mathematics, Columbia University, 2990 Broadway, New York, NY 10027, USA.}
\email{pdaskalo@math.columbia.edu}
\author{Lami Kim}
\address{ {\bf Lami Kim:} Department of Mathematics, Tokyo Institute of Technology, Ookayama, 152-8550, Japan.}
\email{lmkim@math.titech.ac.jp}
\author{Kiahm Lee}
\address{ {\bf Ki-Ahm Lee:} Department of Mathematical Sciences, Seoul National University, Seoul 151-747, Korea, 
\& Center for Mathematical Challenges, Korea Institute for Advanced Study, Seoul 130-722, Korea.}
\email{kiahm@snu.ac.kr}

\begin{document}

\maketitle

\section{Introduction}

We consider a family of complete non-compact strictly convex hypersurfaces $\Sigma_t$ embedded in $\rno$
which  evolve  by  any positive power of their Gauss curvature $K$. 

Given a complete and strict convex hypersurface $\Sigma_0$ embedded in $\rno$, we let $F_0:M^n \rightarrow \rno$ be an immersion  with $F_0(M^n)=\Sigma_0$. 
For a given number $\alpha >0$, we say that an  one-parameter family of immersions $F:M^n\times (0,T) \rightarrow \rno$ is a solution of the 
{\em $\alpha$-Gauss curvature flow}, 
 if for each $t\in (0,T)$, $F(M^n,t)=\Sigma_t$  is a complete and strictly convex hypersurface embedded in $\rno$, and $F(\cdot,t)$ satisfies 
 \begin{equation}\label{eq:INT aGCF}
\begin{cases} 
  \frac{\bd}{\bd t} F(p,t) & = \; K^{\alpha} (p,t)\,  \vec{n}  (p,t) \\
\displaystyle \lim_{t \to 0}  F(p , t ) & = \; F_0(p)
\end{cases} \tag{$*^{\alpha}$}
\end{equation}
where $K(p,t)$ is the Gauss curvature of $\Sigma_t$ at $F(p,t)$ and $\vec{n}(p,t)$ is the interior unit  normal vector of $\Sigma_t$ at the point 
$F(p,t)$.

\medskip

The evolution of compact hypersurfaces by the $\alpha$-Gauss curvature flow is widely studied. 
The case $\alpha=1$ corresponds  to the classical  Gauss curvature
flow which was first introduced by W. Firey in \cite{Fir74GCF} and later studied by 
K. Tso in \cite{Tso85GCF}. Other works include those in  \cite{A99GCF, DH99GCFflat, DL04GCFflat, DSa12PMA, H93GCFflat}. 
The power case was  studied by  B. Chow in \cite{Chow85GCF} and also in the works \cite{A00aGCF, AC11aGCF, ANG15aGCF, KL13aGCF, KLR13aGCFflat}.

In \cite{DSa12PMA} the optimal regularity of viscosity solutions of \eqref{eq:INT aGCF}, which are not necessarily strictly convex, was studied.  These results are local and apply to both compact or non-compact settings. 

\smallskip

In this work we will assume that the initial surface $\Sigma_0$ is a complete  convex graph  over a  domain $\Omega \subset \rn$, that is $\Sigma_0=\{ (x,  u_0(x)): x \in \Omega\}$
for some function $ u_0: \Omega \to \R$. The domain $\Omega$ may be bounded or unbounded.
In particular, the case where $\Sigma_0$ is an entire graph over $\rn$ will be included. 
Since the regularity results of \cite{DSa12PMA} apply in this setting  as well, we will be mostly concerned here
with the questions of existence rather than those of regularity and we have chosen to  take a more classical approach to the problem  and assume  that  the initial hypersurface 
$\Sigma_0$ is {\em  locally uniformly convex}. Our local a priori estimates will  guarantee that the solution is  
locally uniformly convex for all $t >0$ and therefore it becomes instantly  smooth independently from the regularity of the initial data. 
In the case of a compact initial data this result was shown in \cite{A99GCF, AC11aGCF}.

\smallskip

K. Ecker and G. Huisken in  \cite{HE89MCFasymptotic, HE91MCFexist}   studied the evolution of entire graphs
by Mean curvature.    In particular, it was  shown in \cite{HE91MCFexist}  that in some sense the Mean curvature flow  on {\em entire graphs}
behaves better than the  heat equation on $\rn$, namely an entire graph  
solution exists for  all time independently from the growth of the initial surface at infinity. 
The initial entire graph is assumed to be  only locally Lipschitz.  
This result is based on a clever   local gradient estimate which is then  combined with the evolution 
of $|A|^2$ (the square sum of the principal curvatures) to give a  local bound  on $|A|^2$,  independent from the behavior of the
solution at  spatial infinity. 
The latter is achieved  by adopting the  well known technique 
of   Caffarelli, Nirenberg and Spruck  in \cite{CNS88trick} in this geometric setting. More recently,  S{\'a}ez  and Schn{\"u}rer \cite{OS14CpltMean}  showed   the existence of  complete solutions of the Mean curvature flow 
for an  initial hypersurface which is a graph $\Sigma_0=\{(x,u_0(x)):x\in \Omega_0\}$ over a bounded domain  $\Omega_0$,  and $u_0(x) \to +\infty$ as $x \to \bd \Omega_0$. 

\smallskip 
An  open question between the experts in the field is whether the techniques of Ecker and Huisken in
\cite{HE89MCFasymptotic, HE91MCFexist} can be extended to  the fully-nonlinear setting and in particular on entire convex graphs 
evolving  by the  $\alpha$-{\em Gauss curvature flow}. In this work we will show that this is the case as our main result shows.

\begin{theorem} \label{thm:INT Existence}
Let  $\Sigma_0=\{ (x,  u_0(x)): x \in \Omega\}$ be a locally uniformly convex {\em (}defined below{\em )} graph given by a  function $ u_0:\Omega \rightarrow \R$ defined
on a convex open domain $\Omega \subset \rn$ such that  
\begin{enumerate}
\item [{\em (i)}] $ u_0 $ attains its minimum in $\Omega$ and $\inf_{\Omega} u_0 \geq 0$; 
\item [{\em (ii)}] if  $ \Omega \neq \rn$, then  
${\displaystyle  \lim_{x \rightarrow x_0}   u_0(x) = +\infty}$, for all $x_0 \in \bd \Omega$;
\item [{\em (iii)}] if  $\Omega$ is unbounded, then 
${\displaystyle \lim_{R \rightarrow +\infty}\big( \inf_{\Omega \setminus B_R(0)}   u_0 \, \big)  = +\infty}$.
\end{enumerate}

Then, given an immersion $F_0:M^n \to \rno$ such that  $\Sigma_0=F_0(M^n)$,  and for any $\alpha \in (0,+\infty)$, there exists a complete, smooth 
and strictly convex solution $\Sigma_t:=F(M^n,t)$ of the $\alpha$-Gauss curvature flow \emph{(\ref{eq:INT aGCF})} which is defined for all time  $0< t < +\infty$.  
In addition, there is a smooth and strictly convex  function  $ u:\Omega\times (0,+\infty)\rightarrow \mathbb{R}$ satisfying  $\Sigma_t=\{ (x, u(\cdot,t)): \, x \in \Omega\}$, $\displaystyle \lim_{t \rightarrow 0} u(x,t)  = u_0(x)$, and the following 
\begin{equation}\label{eq:INT aGCFeq}
 \frac{\bd u}{\bd t}  = \; \displaystyle \frac{ (\det D^2  u)^\alpha}{(1+|D u|^2)^{\frac{(n+2)\alpha-1}{2}}}. 
\tag{$**^{\alpha}$}
\end{equation}
\end{theorem}
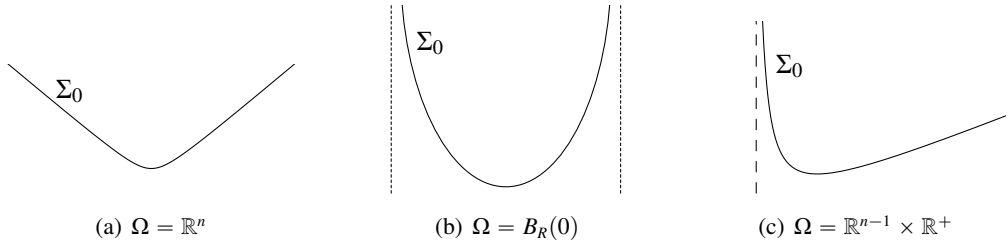
\begin{figure}[h]
\subfigure[$\Omega=\rn$]{
\begin{tikzpicture}[line cap=round,line join=round,>=triangle 45,x=0.18cm,y=0.18cm]
\clip(-11.6278237129,-0.532800819239) rectangle (11.6020914256,9);
\draw [samples=50,domain=-0.99:0.99,rotate around={90.:(0.,0.)},xshift=0.cm,yshift=0.cm] plot ({1.28919490362*(1+(\x)^2)/(1-(\x)^2)},{1.52904430952*2*(\x)/(1-(\x)^2)});
\begin{scriptsize}
\draw[color=black] (-6,7) node[scale=1.3] {$\Sigma_0$};
\end{scriptsize}
\end{tikzpicture}
}
\subfigure[$\Omega = B_R(0)$]{
\begin{tikzpicture}[line cap=round,line join=round,>=triangle 45,x=0.5cm,y=0.5cm]
\clip(-4.67471793468,-1.) rectangle (4.74524982753,4);
\draw [rotate around={90.:(0.,5.)}] (0.,5.) ellipse (2.91547594742cm and 1.4cm);
\draw [dash pattern=on 1pt off 1pt] (3.05,-1.) -- (3.05,5.31432920204);
\draw [dash pattern=on 1pt off 1pt] (-3.05,-1.) -- (-3.05,5.31432920204);
\begin{scriptsize}
\draw[color=black] (-2,3) node[scale=1.3] {$\Sigma_0$};
\end{scriptsize}
\end{tikzpicture}
}
\subfigure[$\Omega = \mathbb{R}^{n-1}\times \mathbb{R}^+$]{
\begin{tikzpicture}[line cap=round,line join=round,>=triangle 45,x=0.3cm,y=0.3cm]
\clip(-2.22055421895,1.36898033791) rectangle (11.0682860864,9);
\draw [samples=50,domain=-0.99:0.99,rotate around={56.309932474:(0.,0.)},xshift=0.cm,yshift=0.cm] plot ({2.9868607644*(1+(\x)^2)/(1-(\x)^2)},{2.01956994781*2*(\x)/(1-(\x)^2)});
\draw [samples=50,domain=-0.99:0.99,rotate around={56.309932474:(0.,0.)},xshift=0.cm,yshift=0.cm] plot ({2.9868607644*(-1-(\x)^2)/(1-(\x)^2)},{2.01956994781*(-2)*(\x)/(1-(\x)^2)});
\draw [dash pattern=on 4pt off 4pt] (0.,1.36898033791) -- (0.,12);
\begin{scriptsize}
\draw[color=black] (1.5,7) node[scale=1.3] {$\Sigma_0$};
\end{scriptsize}
\end{tikzpicture}
}
\caption{Examples of the initial hypersurface $\Sigma_0$}
\end{figure}

Since we have not assumed any regularity on our initial surface $\Sigma_0$, we give next  the definition of a locally uniformly convex hypersurface which is not necessarily of class $C^2$.

\begin{definition}[Uniform convexity for a hypersurface]\label{def:INT Unifrom convexity}
We denote by $C^2_{_{\small \mathcal H}}(\rno)$ the class of second order differentiable complete (either closed or non-compact) hypersurfaces embedded in $\rno$. 
For a  convex hypersurface $\Sigma \in C^2_{_{\mathcal  H} } (\rno)$, we denote by $\lambda_{\min}(\Sigma)(X)$ the smallest principal curvature of $\Sigma$ 
at $X \in \Sigma$.
Given any complete convex hypersurface $\Sigma$ and  a  point $X \in \Sigma$, we define $\lambda_{\min}(\Sigma)(X)$ by
\begin{align*}
\lambda_{\min}(\Sigma)(X) = \sup \big\{ \lambda_{\min}(\Phi)(X): X \in \Phi, \, \Phi \in C^2_{_{\mathcal  H} }  (\rno), \,  \Sigma \subset \text{the convex hull of}\, \Phi  \big\}
\end{align*}
and  we say that 
\begin{enumerate}[(a)]
\item $\Sigma$ is strictly convex, if $\lambda_{\min}(\Sigma)(X)>0$ holds for all $X \in \Sigma$;
\item $\Sigma$ is uniformly convex, if there is a constant $m>0$ satisfying $\lambda_{\min}(\Sigma)(X)\geq m $ for all $X \in \Sigma$;
\item $\Sigma$ is locally uniformly convex, if for any compact subset $A \subset \rno$, there is a constant $m_A>0$ satisfying $\lambda_{\min}(\Sigma)(X)\geq m_A $ 
for all $X \in \Sigma \bigcap A$.
\end{enumerate}
\end{definition}

\begin{remark}[General initial data]
Given a complete and locally uniformly convex  hypersurface $\Sigma_0$ embedded in $\rno$, there exist an orthogonal matrix $A \in O(n+1)$, a vector $Y_0 \in \rno$ and a function $u_0: \Omega \rightarrow \mathbb{R}$ such that $A\Sigma_0+Y_0 \coloneqq \{ Y_0+AX : X \in \Sigma_0 \}=\{ (x,u_0(x)):x \in \Omega\}$, and  the conditions (i)-(iii) 
 in Theorem \ref{thm:INT Existence} hold. Thus, Theorem \ref{thm:INT Existence} shows the existence of a smooth strictly convex solution $\Sigma_t$ of \eqref{eq:INT aGCF} for any complete locally uniformly  convex hypersurface $\Sigma_0$.
\end{remark}

\begin{remark}[Translating solitons and asymptotic behavior]  J. Urbas in  \cite{U98GCFsoliton} studies  complete noncompact convex 
 solutions of \eqref{eq:INT aGCF} which are self-similar.  In particular he shows  that for any 
 $\alpha > 1/2$ and any convex bounded 
 domain $\Omega \subset \R^n$, there exists a translating soliton solution of \eqref{eq:INT aGCF} which is a graph over $\Omega$. 
Conversely,  these are the only complete noncompact convex 
 solutions of \eqref{eq:INT aGCF}  for $\alpha > 1/2$ which move by translation. On the other hand, he shows that
 for $\alpha \in (0,1/2]$  and any $\lambda >0$ there exist a entire graph convex translating soliton solution of  \eqref{eq:INT aGCF}
 which moves with  speed $\lambda$. 
 It would be interesting  to see whether  these translating solitons appear as asymptotic limits,  as $t \to +\infty$,  
 of the solutions  $\Sigma_t$ to  \eqref{eq:INT aGCF} given by  Theorem  \ref{thm:INT Existence} in the corresponding range of exponents. 

\end{remark}

\smallskip

\noindent{\em Discussion on the Proof }:  The proof of Theorem \ref{thm:INT Existence} 
mainly relies on two local a'priori estimates: a local bound from below on the smallest principal 
curvature $\lambda_{\min}$ of the solution $\Sigma_t$,  given in Theorem \ref{thm:LB Local Lower bounds},  
and a local bound  from above on the speed $K^\alpha$ of $\Sigma_t$,  given in Theorem  \ref{thm:SE Speed Estimate}. 
Notice that the first estimate is used to control the speed $K^\alpha$.  The first bound follows  via   a  Pogolerov type computation involving an appropriate  cut-off function on $\Sigma_t$. The second bound uses the well known technique by  Caffarelli, Nirenberg and Spruck in \cite{CNS88trick} which was also used by Ecker and Huisken 
in the context of the Mean curvature flow in  \cite{HE91MCFexist}. 

One of the challenges comes from the fact  that the  evolution equations of the position vector $F(p,t)$ and its curvature have  the differential linearized operators   since $K^\alpha$ is not homogeneous of degree one for $\alpha\neq \frac{1}{n}$. To be specific,  while  the linearized operator $\cL  \coloneqq \frac{\bd K^{\alpha}}{\bd h_{ij}}\nabla_i \nabla_j$ appears in the evolution of curvatures (equations \eqref{eq:Pre aK_t} and \eqref{eq:Pre h_t} below),  
 the position vector $F(p,t)$ , determining the evolution of our cut off functions, 
 satisfies the different  equation $\frac{\bd}{\bd t} F= (n\alpha)^{-1}\cL \, F$. After the linearized operators are matched,  the evolution of  our  cut off function $\psi$ 
 (equation \eqref{eq:Pre psi_t}) has a reaction term which depends on the Gauss curvature. To  establish  local 
 curvature estimates we  carefully  control  the reaction terms which appear from  the gap between the differential operators.
 
 Another difficulty arises  from the non-concavity of the equation \eqref{eq:INT aGCF}. In the most challenging case   where  $n\alpha \geq 1$, \eqref{eq:INT aGCF} is neither concave nor convex equation.  In previous works  \cite{A00aGCF, ANG15aGCF, Chow85GCF, KL13aGCF}
 which concern with the compact case,  global estimates on the smallest principal 
curvature $\lambda_{\min}$  were shown  by using the support function on $S^n$. 

A brief  {\em outline} of the paper  is as follows: in section  \ref{sec-prelim} we   derive some basic equations under the $\alpha$-Gauss curvature flow and also obtain our  local gradient estimate. 
Sections \ref{sec-abounds} and \ref{sec-speed} are devoted to our two crucial local a'priori bounds,  the   lower bound  on the principal curvatures 
and the upper bound  on the speed $K^\alpha$.
In  the  final section we  establish the all time 
existence of the flow which is done in two steps: first we  show  the existence of a complete solution $\Sigma_t$ on $t \in (0,T)$,  where  $T=T_\Omega$ depends on the domain 
$\Omega$.  Then,  we  construct an 
appropriate barrier to guarantee that each $\Sigma_t$ remains  a graph over the same domain $\Omega$,  for all $t \in (0,T)$, 
implying  that $T=+\infty$ independently from the domain $\Omega$.

\subsection{Notation}
For the convenience of the reader, we summarize the following notation which will be frequently used in what follows. 
\begin{enumerate}

\item We recall that $g_{ij} = \langle F_i, F_j \rangle$, where $F_i \coloneqq \nabla_i F$. Also, we denote as usual by $g^{ij}$  the inverse matrix of
$g_{ij}$ and $F^i=g^{ij} \, F_j$.

\item Assume $\Sigma_t$ is a strictly convex graph solution of (\ref{eq:INT aGCF})  in $\rno$. Then, we let $\bar u(\cdot,t):M^n \rightarrow \mathbb{R}$ denote the 
{\em height function}  $\bar u(p,t)=\langle F(p,t),\vec{e}_{n+1} \rangle.$

\item Given  constants $M$ and $\beta \geq 0$, we define the {\em cut-off function}  $\psi_\beta$ by
\begin{align*}
\psi_\beta(p,t)= (M-\beta t-\bar u(p,t))_+ =\max\{M-\beta t-\bar u(p,t) , 0\}.
\end{align*}
In particular, we denote $\psi_0 \coloneqq (M-\bar u(p,t))_+$ by $\psi$ for convenience.

Also, given a constant $R>0$ and a point $Y \in \rno$, we define the {\em cut-off function}  $\bpsi$ by 
\begin{align*}
\bpsi(p,t)= (R^2-|F(p,t)-Y|^2)_+ =\max\{R^2-|F(p,t)-Y|^2 , 0\}
\end{align*}

\item For a strictly convex smooth hypersurface $\Sigma_t$, we denote by $b^{ij}$ the the inverse matrix $(h^{-1})^{ij}$ of its  {\em second fundamental form}  $h_{ij}$, namely $b^{ij}h_{jk}=\delta^i_k$.  Notice that the eigenvalues of $b^{ij}$ on an orthonormal frame are the {\em principal radii of curvature}.

\item We denote by $\cL$  the {\em linearized }   operator 
$$\cL =\alpha  K^{\alpha} b^{ij}\nabla_i \nabla_j$$
Furthermore, $\langle \;, \;\rangle_\cL$ denotes the associated inner product 
$\displaystyle \langle \nabla f,\nabla g \rangle_\cL =  \alpha  K^{\alpha} b^{ij}\nabla_i f \nabla_j g$, 
 where $f,g$ are differentiable functions on $M^n$, and $\|\cdot\|_\cL $ denotes the $\cL$-norm given by the inner product $\langle \;, \;\rangle_\cL$ 

\item We denote by  $\grad=\langle \vec{n},\vec{e}_{n+1} \rangle^{-1}$ the {\em gradient function} (as in \cite{HE91MCFexist}).

\item $H$  denotes the {\em  mean curvature}.
\end{enumerate}

\section{Preliminaries}\label{sec-prelim}

In this section we will derive some basic equations under the $\alpha$-Gauss curvature flow and also obtain a local gradient estimate. 
We begin by showing some basic  evolution equations. 

\begin{proposition} Assume that $\Omega$ and $\Sigma_0$ satisfy the assumptions in  \emph{Theorem \ref{thm:INT Existence}}  and let $\Sigma_t$ be a complete strictly convex   graph solution of \emph{(\ref{eq:INT aGCF})}. Then, the following hold
\begin{align*}
&\bd_t  g_{ij} = - 2 K^{\alpha} h_{ij}  \tag{2.1}\label{eq:Pre g_t}\\
&\bd_t  g^{ij} =  2 K^{\alpha} h^{ij}  \tag{2.2}\label{eq:Pre 1/g_t}\\
&\bd_t  \vec{n} = -\nabla K^{\alpha}\coloneqq -(\nabla_j K^{\alpha}) F^j \tag{2.3}\label{eq:Pre n_t}\\
&\bd_t \psi_\beta =  \cL \,\psi +(n\alpha -1) \, \grad^{-1} K^\alpha-\beta\tag{2.4}\label{eq:Pre psi_t} \\
&\bd_t \bpsi \leq \cL \,\bpsi +2\big(n\alpha+1)( \lambda_{\min}^{-1}+ R)K^{\alpha}\tag{2.5}\label{eq:Pre bpsi_t} \\
&\bd_t h_{ij} = \cL \, h_{ij}+\alpha K^{\alpha}(\alpha b^{kl}b^{mn}-b^{km}b^{nl}) \nabla_i h_{mn}\nabla_j h_{kl}+\alpha K^{\alpha}Hh_{ij}-(1+n\alpha)K^{\alpha} h_{ik}h^k_j \tag{2.6}\label{eq:Pre h_t}\\
&\bd_t K^\alpha = \cL \, K^{\alpha} +\alpha K^{2\alpha }H   \tag{2.7}\label{eq:Pre aK_t}\\
&\bd_t b^{pq}  =\cL  \, b^{pq}-\alpha K^{\alpha}b^{ip}b^{jq}(\alpha b^{kl}b^{mn}+b^{km}b^{nl}) \nabla_i h_{kl}\nabla_j h_{mn}-\alpha K^{\alpha}Hb^{pq}+(1+n\alpha)K^{\alpha}g^{pq}\tag{2.8}
\label{eq:Pre b_t} \\ 
&\bd_t \grad =\cL \, \grad -2\grad^{-1}\|\nabla \grad \|^2_\cL -\alpha K^{\alpha}H\grad\tag{2.9}\label{eq:Pre grad_t}
\end{align*}
\end{proposition}

\smallskip

\noindent\textit{Proof of} (\ref{eq:Pre g_t}).  Observe $\bd_t g_{ij}=\langle \bd_t \nabla_i F,\nabla_j F \rangle+\langle  \nabla_i F, \bd_t \nabla_j F \rangle =\langle  \nabla_i  \bd_t F,\nabla_j F \rangle+\langle  \nabla_i F,  \nabla_j \bd_t F \rangle $. Hence, $\langle  \nabla_i F,  \bd_t F \rangle = 0$ gives $\bd_t g_{ij}= -2 \langle \bd_t F,   \nabla_i \nabla_j F \rangle = -2 \langle K^{\alpha}\vec{n},   h_{ij}\vec{n} \rangle = -2K^{\alpha}h_{ij}$.

\bigskip

\noindent\textit{Proof of} (\ref{eq:Pre 1/g_t}).\,  From $g^{ij}g_{jk}=\delta^i_k$, we can derive $\bd_t g^{ij}=-g^{ik}g^{jl}\bd_t g_{kl}=2K^{\alpha}g^{ik}g^{jl}h_{kl}=2K^{\alpha}h^{ij}$.

\bigskip

\noindent\textit{Proof of} (\ref{eq:Pre n_t}).   $|\vec{n}|^2=1$ implies $\langle \bd_t\vec{n},\vec{n} \rangle =0$. Also, $\langle \vec{n}, \nabla_i F \rangle =0 $ leads to 
\begin{align*}
\langle \bd_t \vec{n},\nabla_i  F\rangle =-\langle \vec{n},\bd_t \nabla_i F \rangle =-\langle \vec{n},\nabla_i \bd_t F \rangle =-\langle \vec{n},\nabla_i (K^{\alpha}\vec{n}) \rangle =-\nabla_i \, K^{\alpha}
\end{align*} 
from which \eqref{eq:Pre n_t} readily follows. 

\bigskip

\noindent\textit{Proof of} (\ref{eq:Pre psi_t}).  The definition of the linearized operator $\cL\eqcolon \alpha K^{\alpha}b^{ij}\nabla_i\nabla_j$ gives
\begin{align*}\label{eq:Pre Position_t}
\cL F \coloneqq \alpha K^{\alpha}b^{ij}\, \nabla_i \nabla_j F=\alpha K^{\alpha}b^{ij}\, h_{ij} \, \vec{n}=n\alpha K^{\alpha}\,  \vec{n}=(n\alpha) \bd_t F\tag{2.10}
\end{align*}
which yields that $\cL \bar u =(n\alpha)\bd_t \bar u $. Therefore, 
\begin{align*}
\bd_t \psi_\beta =-\beta -\bd_t \bar u=  -\cL \bar u +(n\alpha -1) \bd_t \bar u -\beta  
\end{align*}
holds  on the support  of $\psi_\beta\coloneqq(M-\beta t-\bar u(p,t))_+$.
Substituting for   $\bd_t \bar u = \langle \bd_t F ,\vec{e}_{n+1}\rangle = \langle K^{\alpha}\vec{n} ,\vec{e}_{n+1}\rangle=K^{\alpha}\grad^{-1}$,
where $\grad\coloneqq \langle \vec{n} ,\vec{e}_{n+1}\rangle^{-1}$, yields \eqref{eq:Pre psi_t}. 
\bigskip

\noindent\textit{Proof of} (\ref{eq:Pre bpsi_t}). By \eqref{eq:Pre Position_t}, on the support of  $\bpsi\eqcolon (R^2-|F-Y|^2)_+$  we have 
\begin{align*}
\bd_t \bpsi =&  -2\langle (\cL F -(n\alpha)\bd_t F)+\bd_t F,F-Y \rangle =-2\langle \cL F ,F-Y \rangle +2(n\alpha-1)\langle K^{\alpha}\vec{n},F-Y \rangle \\
\leq &\, \cL \,\bpsi +2\|\nabla F \|^2_{\cL}+2K^{\alpha}|n\alpha-1||F-Y| \leq \cL\,\bpsi +2\alpha K^{\alpha}b^{ij}g_{ij}+2(n\alpha+1)RK^{\alpha}\\
\leq &\, \cL \,\bpsi +2n\alpha \lambda_{\min}^{-1}K^{\alpha} +2(n\alpha+1)R K^{\alpha} \leq \cL \,\bpsi +2(n\alpha+1)( \lambda_{\min}^{-1}+R) K^{\alpha}.
\end{align*}

\noindent\textit{Proof of} (\ref{eq:Pre h_t}).  We have 
\begin{align*}
\bd_t h_{ij}  = & \bd_t \, \langle \nabla_i \nabla_j F,\vec{n}\rangle =\langle \nabla_i \nabla_j \bd_t F,\vec{n}\rangle+\langle \nabla_i \nabla_j F,\bd_t \vec{n}\rangle =\langle \nabla_i \nabla_j (K^{\alpha}\vec{n}),\vec{n}\rangle+\langle h_{ij}\, \vec{n},\bd_t \vec{n} \rangle \\
= & \nabla_i \nabla_j K^{\alpha} + \langle \nabla_j K^{\alpha} \nabla_i \vec{n},\vec{n} \rangle + \langle \nabla_i K^{\alpha}\, \nabla_j \vec{n},\vec{n} \rangle +K^{\alpha}\langle \nabla_i \nabla_j \vec{n},\vec{n}\rangle  +0.  
\end{align*}
By $\langle \nabla_i \vec{n},\vec{n} \rangle=0$ and $ \langle \nabla_i \nabla_j \vec{n},\vec{n} \rangle =-\langle  \nabla_j \vec{n},
\nabla_i \vec{n} \rangle =-h_{im}h^m_j$, we deduce 
\begin{align*}\label{eq:Pre h_t with aK}
\bd_t h_{ij} =\nabla_i \nabla_j K^{\alpha} -K^{\alpha} h_{im}h^m_j.  \tag{2.11}
\end{align*}
Applying    
$K=\det(h_{ij}g^{jk})=\det(h_{ij}) \det(g^{kl})$, $\nabla \log\det(h_{ij})= b^{kl}\nabla h_{kl}\, $  and $ {\bd b^{kl}}/{\bd h_{mn}}=-b^{km}b^{ln}$
on the first term on the right hand side of  \eqref{eq:Pre h_t with aK},   yields 
\begin{align*}
\bd_t h_{ij} &= \nabla_i \nabla_j K^{\alpha} -K^{\alpha} h_{im}h^m_j =  \nabla_i (\alpha K^{\alpha}b^{kl} \nabla_j h_{kl}) -K^{\alpha} h_{im}h^m_j \\
&=  \alpha^2K^{\alpha}b^{kl}b^{mn}\nabla_i h_{mn} \nabla_j h_{kl}+\alpha K^{\alpha}\frac{\bd b^{kl}}{\bd h_{mn}}\nabla_i h_{mn} \nabla_j h_{kl}+\alpha K^{\alpha}b^{kl}\nabla_i\nabla_j h_{kl}-K^{\alpha}h_{im}h_{j}^m\\
&=\alpha K^{\alpha}b^{kl}\, \nabla_i\nabla_j h_{kl}+  \alpha K^{\alpha}(\alpha b^{kl}b^{mn}-b^{km}b^{ln})\,  \nabla_i h_{mn}\nabla_j h_{kl}-K^{\alpha}h_{im}h_{j}^m. 
\end{align*}
On the other hand, 
\begin{align*}
\alpha K^{\alpha}b^{kl}\nabla_i\nabla_j h_{kl} = &\alpha K^{\alpha}b^{kl}\nabla_i\nabla_k h_{jl}=\alpha K^{\alpha}b^{kl}\nabla_k\nabla_i h_{jl} +\alpha K^{\alpha}b^{kl}R_{ikjm}h_{l}^m+\alpha K^{\alpha}b^{kl}R_{iklm}h_{j}^m \\
=&\alpha K^{\alpha}b^{kl}\nabla_k\nabla_l h_{ij} +\alpha K^{\alpha}(h_{ij}h_{km}-h_{im}h_{kj})h^m_l b^{lk}+\alpha K^{\alpha}(h_{il}h_{km}b^{kl}-h_{im}h_{kl}b^{kl})h_{j}^m \\
= & \cL h_{ij}+\alpha K^{\alpha}(h_{ij}h_{km}-h_{im}h_{kj})g^{mk}+\alpha K^{\alpha}(h_{im}-n h_{im})\, h_{j}^m \\
= & \cL h_{ij}+\alpha K^{\alpha}H h_{ij}-n\alpha K^{\alpha}h_{im}\, h_{j}^m
\end{align*}
and  (\ref{eq:Pre h_t}) easily  follows.

\bigskip

\noindent\textit{Proof of} (\ref{eq:Pre aK_t}).  From   (\ref{eq:Pre h_t with aK}) we have  
$$\cL K^{\alpha}=\alpha K^{\alpha}b^{ij}( \bd_t h_{ij} + K^{\alpha} h_{im}h^m_j)=\alpha K^{\alpha}b^{ij} \bd_t h_{ij} +\alpha K^{2\alpha} H.$$
 Also, from $K=\det (h_{ij}) \det (g^{kl})$, we derive  that 
\begin{align*}
\bd_t K^{\alpha} &= \alpha K^{\alpha} \bd_t ( \log (\det h_{ij}) + \log (\det g^{kl}))  =  \alpha K^{\alpha}  \, b^{ij}\bd_t h_{ij} + 
\alpha K^\alpha \, g_{kl}\bd_t g^{kl} \\&= \cL K^\alpha - \alpha K^{2\alpha} H + \alpha K^\alpha \, g_{kl}\bd_t g^{kl}.
\end{align*}Aplying  (\ref{eq:Pre 1/g_t}) on the last term yields  (\ref{eq:Pre aK_t}).

\bigskip

\noindent\textit{Proof of} (\ref{eq:Pre b_t}). The identity $b^{ij}h_{jk}=\delta^i_k$ leads to  
\begin{align*}
\bd_t b^{pq} & =  -b^{ip}b^{jq}\bd_t h_{ij} \tag{2.12}\label{eq:Pre Inverse speed}\\
\nabla_m b^{pq} &  = -b^{ip}b^{jq}\nabla_m h_{ij}\tag{2.13}\label{eq:Pre Inverse gradient}
\end{align*}
Therefore, 
\begin{align*}
\nabla_n\nabla_m b^{pq} &    =-b^{jq}\nabla_n b^{ip}\nabla_m h_{ij}-b^{ip}\nabla_nb^{jq}\nabla_m h_{ij}-b^{ip}b^{jq}\nabla_n\nabla_m h_{ij}  \\
&  =2b^{jq}b^{ik}b^{pl}\nabla_n h_{kl}\nabla_m h_{ij}-b^{ip}b^{jq}\nabla_n\nabla_m h_{ij}.
\end{align*}
Hence,   $\cL b^{pq} \coloneqq \alpha K^{\alpha}\, b^{nm}\, \nabla_n\nabla_m b^{pq}$ satisfies 
\begin{align*}
\cL b^{pq} = 2\alpha K^{\alpha}b^{nm}b^{jq}b^{ik}b^{pl}\nabla_n h_{kl}\nabla_m h_{ij}-b^{ip}b^{jq}\cL h_{ij}  = 2\alpha K^{\alpha}b^{lp}b^{jq}b^{ki}b^{nm}\nabla_l h_{kn}\nabla_j h_{im}-b^{ip}b^{jq}\cL h_{ij}.
\end{align*}
Combing  the last identity with  (\ref{eq:Pre h_t}) and (\ref{eq:Pre Inverse speed}) yields 
\begin{align*}
\bd_t b^{pq} & = -b^{ip}b^{jq}(\cL h_{ij}+\alpha K^{\alpha}(\alpha b^{kl}b^{mn}-b^{km}b^{nl}) \nabla_i h_{kl}\nabla_j h_{mn}+\alpha K^{\alpha}H\, h_{ij}-(1+n\alpha)K^{\alpha} h_{ik}h^k_j) \\
&=\cL  b^{pq}-\alpha K^{\alpha}b^{ip}b^{jq}(\alpha b^{kl}b^{mn}+b^{km}b^{nl}) \nabla_i h_{kl}\nabla_j h_{mn}-\alpha K^{\alpha}Hb^{pq}+(1+n\alpha)K^{\alpha}g^{pq}
\end{align*}
which gives \eqref{eq:Pre b_t}. 
\bigskip

\noindent\textit{Proof of} (\ref{eq:Pre grad_t}).  By (\ref{eq:Pre n_t}) we have $\bd_t \grad  = -\grad^2 \, \langle \vec{e}_{n+1},\bd_t \vec{n} \rangle =\grad^2\, \langle \vec{e}_{n+1},\nabla K^{\alpha} \rangle$. Furthermore,
\begin{align*}
\cL \grad & = -\alpha K^{\alpha}b^{ij}\, \nabla_i (\grad^2 \, \langle \vec{e}_{n+1},\nabla_j \vec{n}\rangle) =-2\alpha K^{\alpha}b^{ij}\grad
\,  \nabla_i\grad\,  \langle \vec{e}_{n+1},\nabla_j \vec{n}\rangle +\alpha K^{\alpha}b^{ij}\grad^2 \, \langle \vec{e}_{n+1},\nabla_i(h_{jk}F^k)\rangle \\
& = 2 \grad^{-1}\|\nabla \grad\|^2_{\cL}  +\alpha K^{\alpha}b^{ij}h_{jk}h_i^k\, \grad^2 \, \langle \vec{e}_{n+1},\vec{n}\rangle+ \grad^2\,  \langle \vec{e}_{n+1},\alpha K^{\alpha}b^{ij} \, (\nabla_ih_{jk})F^k\rangle \\
& = 2 \grad^{-1}\|\nabla \grad\|^2_{\cL}  +\alpha K^{\alpha} H \grad+\grad^2\langle \vec{e}_{n+1},\nabla K^{\alpha} \rangle
\end{align*}
Combining the above yields  (\ref{eq:Pre grad_t}).

%

\bigskip
 We recall the notation $\grad \coloneqq \langle \vec{n},\vec{e}_{n+1} \rangle^{-1}$ (gradient function)  and 
$\psi_\beta(p,t)\coloneqq(M-\beta t-\bar u(p,t))_+$ (cut-off function) where $ \bar u(p,t)=\langle F(p,t),\vec{e}_{n+1} \rangle$ denotes the height function. 
We will next show the following local gradient estimate. 

\begin{theorem} [Gradient estimate] \label{thm:Pre Gradient}
Assume that  a smooth hypersurface $\Sigma_0$ satisfies the assumptions in \emph{Theorem \ref{thm:INT Existence}}  and let $\Sigma_t$ be a complete strictly convex smooth graph solution of \emph{(\ref{eq:INT aGCF})} defined on $M^n \times [0,T]$,  for some $T>0$.   Given  constants $\beta >0$ and $M\geq \beta$,
\begin{equation*}
\grad(p,t) \, \psi_\beta(p,t) \leq M\, \max \big\{\,\sup_{Q_M}\grad (p,0) ,\,  \beta^{-1}n^{\frac{1}{n\alpha+1}}(n\alpha-1)_+\big\}
\end{equation*}
where $Q_M=\{p\in M^n: \bar u(p,0) < M\}$.
\end{theorem}

\begin{proof}
First use  (\ref{eq:Pre psi_t}) and (\ref{eq:Pre grad_t}),  that is 
\begin{align*}
\bd_t \psi_\beta& =  \cL \psi_\beta +(n\alpha -1)\grad^{-1} K^\alpha-\beta \\
\bd_t \grad & =\cL \grad -2\grad^{-1}\|\nabla \grad \|_\cL^2 -\alpha K^{\alpha}H\grad 
\end{align*}
to compute 
\begin{align*}
\bd_t (\psi_\beta \grad) & =\psi_\beta \cL \grad -2\psi_\beta\grad^{-1}\|\nabla \grad \|_\cL^2 -\alpha \psi_\beta K^{\alpha}H\grad +\grad \cL \psi_\beta +(n\alpha -1) K^\alpha-\beta\grad  \\
& = \cL (\psi_\beta \grad)  -2\langle \nabla \psi_\beta ,\nabla \grad \rangle_\cL -2\psi_\beta\grad^{-1}\|\nabla \grad \|_\cL^2 -\alpha \psi_\beta K^{\alpha}H\grad  +(n\alpha -1) K^\alpha-\beta\grad \\
& = \cL (\psi_\beta \grad)  -2\grad^{-1}\langle \nabla (\psi_\beta\grad) ,\nabla \grad \rangle_\cL  -\alpha \psi_\beta K^{\alpha}H\grad  +(n\alpha -1) K^\alpha-\beta\grad. 
\end{align*}
Since (ii), (iii) in Theorem \ref{thm:INT Existence} imply  that $\psi_\beta$ is compactly supported, for a fixed $T \in (0,+\infty)$, the function $\psi_\beta\, \grad$ attains its maximum on  $M^n \times [0,T] $ at some $(p_0,t_0)$. If $t_0 =0$, then we obtain the desired result. Assume that $t_0 > 0$. Then, at $(p_0,t_0)$, we have
$$\bd_t (\psi_\beta \grad) - \cL (\psi_\beta \grad) \geq 0$$
which is equivalent to  
\begin{align*}
(n\alpha -1) \, K^\alpha \geq \alpha \psi_\beta K^{\alpha} H \, \grad  +\beta \grad =(\alpha \psi_\beta K^{\alpha}H+\beta)\, \grad. 
\end{align*}
Hence, an  interior maximum can be achieved only if $n\alpha \geq 1$. From now on, we assume that $n\alpha \geq 1$.  If we  multiply 
the last inequality by $MK^{-\alpha}$ and use that $M\geq \beta$ we obtain 
\begin{align*}
 (n\alpha -1)M \geq (\alpha M \psi_\beta H+\beta M K^{-\alpha})\grad  \geq  \beta(\alpha  \psi_\beta H+M K^{-\alpha})\grad. 
\end{align*}
On the other hand,  $M \geq \psi_\beta$ implies  $\alpha  \psi_\beta H+M K^{-\alpha}  \geq   \alpha  \psi_\beta H+ \psi_\beta K^{-\alpha}\geq (\alpha  H +K^{-\alpha})\psi_\beta $. Hence, 
\begin{align*}
(n\alpha -1)M \geq \beta(\alpha  H +K^{-\alpha})\psi_\beta \grad  \geq \frac{\beta}{n}\Big (\frac{n\alpha}{n\alpha+1}H+\frac{1}{n\alpha+1} K^{-\alpha} \Big )\, \psi_\beta \grad. 
\end{align*}
Next, we apply the Young's inequality
\begin{align*}
\frac{n\alpha}{n\alpha+1}H+\frac{1}{n\alpha+1} K^{-\alpha} \geq H^{\frac{n\alpha}{n\alpha+1}}K^{-\frac{\alpha}{n\alpha+1} }=(HK^{-\frac{1}{n}})^{\frac{n\alpha}{n\alpha+1} }. 
\end{align*}
Since  $H \geq n\, K^{{1}/{n}} $, we conclude that in the case $n\alpha \geq 1$, $t_0 >0$, the following holds  at $(p_0,t_0)$
\begin{align*}
(n\alpha -1)_+M \geq \beta n^{-1}n^{\frac{n\alpha}{n\alpha+1} } \psi_\beta \grad =\beta n^{-\frac{1}{n\alpha+1} }\, \psi_\beta \grad 
\end{align*}
from which the desired inequality readily follows.

\end{proof}

\section{Lower bounds on the principal curvatures}\label{sec-abounds}

In this section we will obtain  lower bounds  on the principal curvatures, as stated in Theorem \ref{thm:LB Local Lower bounds}.
We will achieve this by using  Pogorelov type estimates with respect to $b^{ii}$. Recall that $b^{ij}$ denotes the inverse matrix
of the second fundamental form $h_{ij}$. Since $b^{ii}$ depends on charts, we will  find the relation between $b^{ii}$ and the 
principal curvatures as  \cite{CD15Qk}. We begin with a simple observation 
holding on every smooth strictly convex hypersurface which we prove here for the reader's convenience. 

\begin{proposition}[Euler's formula] \label{prop:LB Euler's formula} Let $\Sigma$ be a smooth strictly convex hypersurface  and $F:M^n \rightarrow \rno$ be a smooth immersion with $F(M^n)=\Sigma$. Then, for all $p \in M^n$ and $i \in \{1,\cdots,n\}$, the following holds
\begin{align*}
\frac{b^{ii}(p)}{g^{ii}(p)} \leq \frac{1}{\lambda_{\min}(p)}. 
\end{align*}
\end{proposition}

\begin{proof}
Let  $\{ E_1, \cdots,E_n \}$ be an  orthonormal basis of $T\Sigma_{F(p)}$ satisfying $L(E_j) =\lambda_j\, E_j $,
where $L$ is the Weingarten map  and $\lambda_1,\cdots,\lambda_n$ are the principal curvatures of $\Sigma$ at $p$. 
Denote by $\{ a_{ij}\}$  the matrix satisfying $F_i(p)\coloneqq \nabla_i F(p)  =a_{ij}E_j$  and by $\{c_{ij}\}$  the diagonal matrix  $\text{diag}(\lambda_1,\cdots,\lambda_n)$. Then, $   L F_i(p)=h_{ij}(p) F^j(p)  $  implies
\begin{align*}
b^{ij}(p)LF_j(p)=b^{ij}(p)h_{jk}(p)F^k(p)= F^i(p)  = g^{ij}(p)a_{jm}E_m
\end{align*}
Observing that for the sum  $ a_{jk}E_k = \sum_{k=1}^n a_{jk}E_k$ we have 
$\displaystyle L(a_{jk}E_k)=a_{jk}L(E_k)=a_{jk}\lambda_kE_k=a_{jk}c_{km}E_m$, it follows that 
\begin{align*}
g^{ij}(p)a_{jm}E_m =b^{ij}(p)LF_j(p)=b^{ij}(p)L(a_{jk}E_k)=b^{ij}(p)a_{jk}c_{km}E_m
\end{align*}
Hence, 
$$g^{ij}(p)a_{jm}=b^{ij}(p)a_{jk}c_{km}$$
Denoting by  $a^{ij}$ and $c^{ij}$ the inverse matrices of $a_{ij}$ and $c_{ij}$  respectively, we have
\begin{equation*}
g^{ij}(p)a_{jm}c^{ml}a^{nl}=b^{ij}(p)a_{jk}c_{km}c^{ml}a^{nl}=b^{in}(p)
\end{equation*}  
Thus, for each  $i \in \{1,\cdots,n\}$, the following holds
\begin{align*}
b^{ii}(p)&=\sum_{j,l,m}g^{ij}(p)a_{jm}c^{ml}a^{il} = \sum_{j,l,m}\langle  F^i(p),F^j(p)\rangle a_{jm}c^{ml}a^{il} 
\end{align*}
By setting $F^i(p)=\bar a_{ij}E_j$, we observe $a^{ij}=a^{ik}\langle F_k(p), F^j(p)\rangle=a^{ik}\langle a_{kl}E_l,\bar a_{jm}E_m \rangle =a^{ik}a_{kl} \bar a_{jm} \delta_{ml}=\bar a_{ij} $. Thus, we have $\langle  F^i(p),F^j(p)\rangle=\langle  a^{ik}E_k,a^{jr}E_r\rangle=a^{ik}a^{kj}$, which yields
\begin{align*}
b^{ii}(p)&=\sum_{j,k,l,m} a^{ik}a^{jk}a_{jm}c^{ml}a^{il} =\sum_{k,l,m} a^{ik}\delta^{k}_{m}c^{ml}a^{il}=\sum_{k,l} a^{ik}c^{kl}a^{il}=\sum_{k} a^{ik}a^{ik}\lambda^{-1}_k =\sum_{k} (a^{ik})^2\lambda^{-1}_k \\
& \leq \sum_{k} (a^{ik})^2\lambda^{-1}_{\min} = \lambda^{-1}_{\min}\sum_{k,j} \langle a^{ik}E_k,a^{ij}E_j \rangle = \lambda^{-1}_{\min} \langle F^i(p),F^i(p) \rangle = \lambda^{-1}_{\min}g^{ii}(p)
\end{align*}
\end{proof}

\begin{remark}
In \cite{Chow85GCF}, Chow obtained  lower bounds on the  principal curvatures of a closed solution $\Sigma_t$ of ($*^{1/n}$) by applying the maximum principle for $ \bar \nabla_{\xi}\bar \nabla_{\xi} S + \bar g_{\xi\xi} S$ where $S(\cdot,t)$ is the support function of $\Sigma_t$, and $\bar \nabla$ is a covariant derivative compatible with a metric $\bar g$ of the unit sphere $S^n$. See also \cite{A00aGCF, ANG15aGCF, KL13aGCF}. Notice that the eigenvalues of $\bar \nabla_i \bar \nabla_j S + \bar g_{ij} S$ with respect to $\bar g_{ij}$ are the principal radii of curvature as those of $b^{ij}$, and the maximum principle also apply for $b^{\xi\xi}/g^{\xi\xi}$ if $\Sigma_t$ is a closed solution. However, $g_{ij}$ depends on time $t$ while $\bar g_{ij}$ does not.
\end{remark}

We will next show a local  lower bound on  the principal curvatures of
$\Sigma_t$ in terms of the initial data, which  constitutes one of the crucial estimates in this  work. We recall the definition of the cut-off function   
$\psi_\beta(p,t)\coloneqq(M-\beta t-\bar u(p,t))_+$,  where   $ \bar u(p,t) \coloneqq\langle F(p,t),\vec{e}_{n+1} \rangle$  denotes the height function.

\begin{theorem}[Local lower bound for the principal curvatures]\label{thm:LB Local Lower bounds}
Assume that  a smooth hypersurface $\Sigma_0$ satisfies the assumptions in \emph{Theorem \ref{thm:INT Existence}}  and let $\Sigma_t$ be a complete strictly convex smooth graph solution of \emph{(\ref{eq:INT aGCF})} defined on $M^n \times [0,T]$,  for some $T>0$.   Then, given  
constants $\beta >0$ and $M \geq \beta$, the following holds
\begin{equation*}
\big ( \psi_\beta^{-n(1+\frac{1}{\alpha})} \lambda_{\min} \big )(p,t) \geq M^{-n(1+\frac{1}{\alpha})}\, \min \bigg \{ \inf_{ Q_M}
  \lambda_{\min}(p,0)  , \frac{\beta^{(n-1)+\frac{1}{\alpha}} }{( n^2(n+1)(n\alpha -1)_+)^{(n-1)+\frac{1}{\alpha}}} \bigg \}
\end{equation*}
where $Q_M=\{p\in M^n: \bar u(p,0) < M\}$, and $(n\alpha-1)^{-1}_+ = +\infty$, if $n\alpha \leq 1$.
\end{theorem}

\begin{proof}
Consider the cut-off function $\psi_\beta\coloneqq  (M-\beta t-\bar u(p,t))_+$, where  $ \bar u(p,t)=\langle F(p,t),\vec{e}_{n+1} \rangle$ denotes the 
height function. Observe that  the conditions (ii), (iii) in Theorem \ref{thm:INT Existence} imply  that $\psi_\beta$ is compactly supported.
Therefore,  for a fixed $T \in (0,+\infty)$, the function $\psi_\beta^{n(1+\frac{1}{\alpha})}\lambda_{\min}^{-1}$ 
attains its maximum on  $M^n \times [0,T]$ at 
a  point $(p_0,t_0)$.  If $t_0 =0$, then we obtain the desired result by the bound 
$\psi_\beta \leq M $  and the conditions on our initial data. So, we may assume in what follows that $t_0 >0$.

We begin by choosing a chart $(U,\varphi)$ with $p_0 \in \varphi(U) \subset M^n$, on which  the covariant derivatives $\big\{ \nabla_i F(p_0,t_0)\coloneqq\bd_i (F \circ \varphi)(\varphi^{-1}(p_0),t_0) \big \}_{i=1,\cdots,n}$ form an orthonormal basis of $(T\Sigma_{t_0})_{F(p_0,t_0)}$ satisfying
$$
g_{ij}(p_0,t_0)=\delta_{ij}, \qquad  h_{ij}(p_0,t_0)=\delta_{ij}\lambda_i(p_0,t_0), \qquad  \lambda_1(p_0,t_0) =\lambda_{\min}(p_0,t_0).
$$
In particular, at the point $(p_0,t_0)$ we have   $b^{11}(p_0,t_0) = \lambda^{-1}_{\min}(p_0,t_0) $ and $g^{11}(p_0,t_0) = 1$.
Next, we  define the  function $w: \varphi(U) \times [0,T] \rightarrow \mathbb{R}$ by
\begin{align*}
w\coloneqq\psi_\beta^{n(1+\frac{1}{\alpha})}\frac{ b^{11}}{g^{11}}. 
\end{align*}
Notice that on the chart $(U,\varphi)$,  if $t\neq t_0$, then the covariant derivatives $\big\{\nabla_i F(p_0,t)\big\}_{i=1,\cdots,n} $ may not form an orthonormal basis of $(T\Sigma_{t})_{F(p_0,t)}$. However, since Proposition \ref{prop:LB Euler's formula} holds for every chart and immersion, we 
have  $w \leq \psi_\beta^{n(1+\frac{1}{\alpha})}\lambda_{\min}^{-1}$. Hence, for $(p,t) \in \varphi( U) \times [0,T]$, the following holds
\begin{align*}
w(p,t) \leq \psi_\beta^{n(1+\frac{1}{\alpha})}\lambda_{\min}^{-1}(p,t)  \leq  \psi_\beta^{n(1+\frac{1}{\alpha})}\lambda_{\min}^{-1}(p_0,t_0)=w(p_0,t_0)
\end{align*}
which shows that  $w$ attains its maximum at $(p_0,t_0)$. 

Observe next that since  $\nabla g^{11}=0$, the following holds  on the support of $\psi_\beta$
\begin{align*}\label{eq:LB Pinching Gradeint}
\frac{\nabla_i w}{w} =  n\Big (1+\frac{1}{\alpha} \Big)\frac{\nabla_i \psi_\beta}{\psi_\beta}+ \frac{\nabla_i b^{11}}{b^{11}}. \tag{3.1}
\end{align*}
Let us differentiate the equation above, again.
\begin{align*}
\frac{\nabla_i \nabla_j w}{w}-\frac{\nabla_i  w \nabla_j w}{w^2} = n\Big(1+\frac{1}{\alpha}\Big)\frac{\nabla_i\nabla_j \psi_\beta}{\psi_\beta}-  n\Big(1+\frac{1}{\alpha}\Big)\frac{\nabla_i \psi_\beta \nabla_j \psi_\beta}{\psi_\beta^2}+\frac{\nabla_i\nabla_j b^{11}}{b^{11}}-\frac{\nabla_i b^{11}\nabla_j b^{11}}{(b^{11})^2}. 
\end{align*}
Multiply by $\alpha K^{\alpha}\, b^{ij}$  and sum the equations over all $i,j$ to obtain 
\begin{align*}
\frac{\cL w}{w}-\frac{\|\nabla  w \|^2_{\cL}}{w^2} =  n\Big(1+\frac{1}{\alpha}\Big)\frac{\cL \psi_\beta}{\psi_\beta}- n\Big(1+\frac{1}{\alpha}\Big)\frac{\|\nabla \psi_\beta \|^2_{\cL}}{\psi_\beta^2}+\frac{\cL b^{11}}{b^{11}}-\frac{\|\nabla  b^{11}\|^2_{\cL}}{(b^{11})^2}. 
\end{align*}
On the other hand, on the support of $\psi_\beta$, we also have
\begin{align*}
\frac{\bd_t w}{w} = n\Big(1+\frac{1}{\alpha}\Big)\frac{\bd_t \psi_\beta}{\psi_\beta}+ \frac{\bd_t b^{11}}{b^{11}}-\frac{\bd_t g^{11}}{g^{11}}.  
\end{align*}
Subtract the equations above. Then, $w^{-2}\|\nabla  w \|^2_{\cL}\geq 0$ implies the following inequality
\begin{align*}
\frac{\cL w}{w}-\frac{\bd_t w}{w} \geq   n\Big(1+\frac{1}{\alpha}\Big)\Big(\frac{\cL \psi_\beta}{\psi_\beta}-\frac{\bd_t \psi_\beta}{\psi_\beta}\Big)- n\Big(1+\frac{1}{\alpha}\Big)\frac{\|\nabla \psi_\beta \|^2_{\cL}}{\psi_\beta^2}+\Big(\frac{\cL b^{11}}{b^{11}}-\frac{\bd_t b^{11}}{b^{11}}\Big)-\frac{\|\nabla  b^{11}\|^2_{\cL}}{(b^{11})^2}+\frac{\bd_t g^{11}}{g^{11}}.  
\end{align*}
By  \ref{eq:Pre 1/g_t} and \eqref{eq:Pre psi_t} we have $ \bd_t g^{11} = 2K^{\alpha} h^{11}$ and  
$\cL \psi_\beta - \bd_t \psi_\beta  = \beta - (n\alpha -1)\grad^{-1} K^\alpha$,   while by \eqref{eq:Pre b_t}
\begin{align*}
\cL  b^{11}  - \bd_t b^{11}  = \alpha K^{\alpha}b^{i1}b^{j1}(\alpha b^{kl}b^{mn} + b^{km}b^{nl}) \nabla_i h_{kl}\nabla_j h_{mn}+\alpha K^{\alpha}Hb^{11}-(1+n\alpha)K^{\alpha}g^{11}. 
\end{align*}
Combining the above yields 
\begin{align*}\label{eq:LB Pinching 1st}
\frac{\cL w}{w}-\frac{\bd_t w}{w} \geq     -& n\Big(1+\frac{1}{\alpha}\Big)\frac{\|\nabla \psi_\beta \|^2_{\cL}}{\psi_\beta^2}-\frac{\|\nabla  b^{11}\|^2_{\cL}}{(b^{11})^2}+\frac{\alpha K^{\alpha}b^{i1}b^{j1}(\alpha b^{kl}b^{mn}+b^{km}b^{nl})}{b^{11}} \nabla_i h_{kl}\nabla_j h_{mn}\\
&-n\Big(1+\frac{1}{\alpha}\Big)(n\alpha -1)\frac{K^\alpha \grad^{-1}}{ \psi_\beta}+n\Big(1+\frac{1}{\alpha}\Big)\frac{\beta}{\psi_\beta}  +\alpha K^{\alpha}H-(1+n\alpha)K^{\alpha}\frac{g^{11}}{b^{11}}+2K^{\alpha}\frac{h^{11}}{g^{11}}. \tag{3.2}
\end{align*}
Now, at $(p_0,t_0)$, the following holds 
\begin{align*}\label{eq:LB pinching reaction}
\alpha K^{\alpha}H-(1+n\alpha)K^{\alpha}\frac{g^{11}}{b^{11}}+2K^{\alpha}\frac{h^{11}}{g^{11}} 
\geq n\alpha K^{\alpha}\lambda_{\min}-(1+n\alpha)K^{\alpha}\lambda_{\min}+2K^{\alpha}\lambda_{\min}=K^{\alpha}\lambda_{\min}. \tag{3.3}
\end{align*}
In addition, if   $n\alpha \geq 1$, then $H \geq n \lambda_{\min}$ gives 
\begin{align*}
\alpha K^{\alpha}H=(\alpha-\frac{1}{n}) K^{\alpha}H +\frac{1}{n} K^{\alpha}H \geq (n\alpha-1) K^{\alpha}\lambda_{\min} +\frac{1}{n} K^{\alpha}H. 
\end{align*}
Therefore, in the case that  $n\alpha \geq 1$, we can improve (\ref{eq:LB pinching reaction}) to obtain 
\begin{align*}\label{eq:LB pinching reaction second}
\alpha K^{\alpha}H-(1+n\alpha)K^{\alpha}\frac{g^{11}}{b^{11}}+2K^{\alpha}\frac{h^{11}}{g^{11}}
\geq \frac{1}{n} K^{\alpha}H.  \tag{3.4}
\end{align*}
Also, at the maximum point  $(p_0,t_0)$ of $w$,  $\nabla w (p_0,t_0)=0$ holds. So, (\ref{eq:LB Pinching Gradeint}) leads to
\begin{align*}
 \frac{n(1+\alpha)}{\alpha}\frac{\|\nabla \psi_\beta \|^2_{\cL}}{\psi_\beta^2}+\frac{\|\nabla  b^{11}\|^2_{\cL}}{(b^{11})^2} 
=\Big(1+\frac{\alpha}{n(1+\alpha)}\Big)\frac{\|\nabla  b^{11}\|^2_{\cL}}{(b^{11})^2} =\Big(1+\frac{\alpha}{n(1+\alpha)}\Big)\alpha\sum_{i=1}^n\frac{b^{ii}K^{\alpha}|\nabla_i b^{11}|^2}{(b^{11})^2}. 
\end{align*}
From (\ref{eq:Pre Inverse gradient}), we get $\nabla_i b^{11}  = -b^{1k}b^{1l}\nabla_i h_{kl}=-(b^{11})^2\nabla_i h_{11}$ at $(p_0,t_0)$. Hence,
\begin{align*}
 \frac{n(1+\alpha)}{\alpha}\frac{\|\nabla \psi_\beta \|^2_{\cL}}{\psi_\beta^2}+\frac{\|\nabla  b^{11}\|^2_{\cL}}{(b^{11})^2} 
=\alpha \Big(1+\frac{\alpha}{n(1+\alpha)}\Big)\sum_{i=1}^n b^{ii}(b^{11})^2K^{\alpha}|\nabla_i h_{11}|^2. 
\end{align*}
Defining,  at $(p_0,t_0)$, we define 
\begin{align*}
I_i &= b^{ii}(b^{11})^2 K^{\alpha}|\nabla_i h_{11}|^2 \\
J_i & = b^{ii}\nabla_1 h_{ii}. 
\end{align*}
We may  rewrite the equation above as
\begin{align*}\label{eq:LB 3rd-I}
 \frac{n(1+\alpha)}{\alpha}\frac{\|\nabla \psi_\beta \|^2_{\cL}}{\psi_\beta^2}+\frac{\|\nabla  b^{11}\|^2_{\cL}}{(b^{11})^2} 
=\alpha(1+\frac{\alpha}{n(1+\alpha)})\sum_{i=1}^n I_i\tag{3.5}
\end{align*}
and also  write 
\begin{align*}
 b^{kl}b^{mn}  \nabla_1 h_{kl}\nabla_1 h_{mn} = |b^{mn}\nabla_1 h_{mn}|^2 = \Big|\sum^n_{i=1} b^{ii}\nabla_1 h_{ii}\Big|^2=\Big|\sum^n_{i=1} J_i\Big|^2
\end{align*}
which gives  
\begin{align*}
\frac{\alpha^2 K^{\alpha}b^{i1}b^{j1} b^{kl}b^{mn}}{b^{11}} \nabla_i h_{kl}\nabla_j h_{mn}&=\alpha^2 K^{\alpha}b^{11} b^{kl}b^{mn}\nabla_1 h_{kl}\nabla_1 h_{mn}=\alpha^2K^{\alpha}b^{11} \Big|\sum_{i=1}^nJ_i \Big|^2. 
\end{align*}
Since ${\displaystyle \alpha^2 \geq \frac{\alpha^2}{1+\alpha}}$, we  conclude that 
\begin{align*}\label{eq:LB 3rd-II}
\frac{\alpha^2 K^{\alpha}b^{i1}b^{j1} b^{kl}b^{mn}}{b^{11}} \nabla_i h_{kl}\nabla_j h_{mn} 
\geq \frac{\alpha^2}{1+\alpha}K^{\alpha}b^{11}\Big|\sum_{i=1}^nJ_i \Big|^2.  \tag{3.6}
\end{align*}
Finally, at $(p_0,t_0)$, we also have 
\begin{align*}
\frac{\alpha K^{\alpha}b^{i1}b^{j1}b^{km}b^{nl}}{b^{11}} \nabla_i h_{kl}\nabla_j h_{mn}& = \alpha K^{\alpha}b^{11}b^{km}b^{nl} \nabla_1 h_{kl}\nabla_1 h_{mn}\\
&=\alpha K^{\alpha}b^{11}(\sum^n_{i=1} |b^{ii}\nabla_1 h_{ii}|^2 +\sum_{i\neq j}b^{ii}b^{jj}|\nabla_{1}h_{ij}|^2) \\
&\geq \alpha K^{\alpha}b^{11}(\sum^n_{i=1} |b^{ii}\nabla_1 h_{ii}|^2 +2\sum_{i\neq 1}b^{ii}b^{11}|\nabla_{1}h_{i1}|^2) \\
& =\alpha K^{\alpha}b^{11}\sum^n_{i =  1} |J_i|^2 +2\alpha\sum_{i\neq 1}I_i. 
\end{align*}
Using  ${\displaystyle \alpha \geq \frac{\alpha^2}{1+\alpha}}$, we may rewrite the inequality above as
\begin{align*}\label{eq:LB 3rd-III}
\frac{\alpha K^{\alpha}b^{i1}b^{j1}b^{km}b^{nl}}{b^{11}} \nabla_i h_{kl}\nabla_j h_{mn} \geq \frac{\alpha^2}{1+\alpha}K^{\alpha}b^{11}\sum_{i\neq  1} |J_i|^2+\alpha K^{\alpha}b^{11} |J_1|^2 +2\alpha\sum_{i\neq 1}I_i. \tag{3.7}
\end{align*}
Adding  (\ref{eq:LB 3rd-II}) and (\ref{eq:LB 3rd-III}), gives that at $(p_0,t_0)$ we have 
$$
\frac{\alpha K^{\alpha}b^{i1}b^{j1}(\alpha b^{kl}b^{mn}+b^{km}b^{nl})}{b^{11}} \nabla_i h_{kl}\nabla_j h_{mn} \geq 
\frac{\alpha^2}{1+\alpha}K^{\alpha}b^{11} \left ( \Big|\sum_{i=1}^nJ_i \Big|^2 + \sum_{i\neq  1} |J_i|^2 \right )
+\alpha K^{\alpha}b^{11} |J_1|^2 +2\alpha\sum_{i\neq 1}I_i
$$
and by  the Cauchy-Schwarz inequality 
\begin{align*}
n \left(\Big|\sum_{i=1}^n J_i \Big|^2+\sum_{i \neq 1} |J_i|^2 \right)& =(1^2+(-1)^2+\cdots+(-1)^2)\left (\Big|\sum_{i=1}^n J_i \Big|^2+\sum_{i \neq 1} |J_i|^2 \right)\\
& \geq \Big |\sum_{i=1}^n J_i +\sum_{i\neq 1} -J_i  \, \Big |^2=|J_1|^2
\end{align*}
and ${\displaystyle 2\alpha \geq \alpha(1+\frac{\alpha}{n(1+\alpha)})}$, we obtain 
\begin{align*}
\frac{\alpha K^{\alpha}b^{i1}b^{j1}(\alpha b^{kl}b^{mn}+b^{km}b^{nl})}{b^{11}} \nabla_i h_{kl}\nabla_j h_{mn} \geq \alpha \Big(1+\frac{\alpha}{n(1+\alpha)}\Big)
\Big ( K^{\alpha}b^{11} |J_1|^2 + \sum_{i\neq 1}I_i \Big ). 
\end{align*}
Combining the last inequality with (\ref{eq:LB 3rd-I}) while using  that  $K^{\alpha}b^{11} |J_1|^2=I_1$,   yields 
\begin{align*}\label{eq:LB 3rd-IV}
-\frac{n(1+\alpha)}{\alpha}\frac{\|\nabla \psi_\beta \|^2_{\cL}}{\psi_\beta^2}-\frac{\|\nabla  b^{11}\|^2_{\cL}}{(b^{11})^2} +\frac{\alpha K^{\alpha}b^{i1}b^{j1}(\alpha b^{kl}b^{mn}+b^{km}b^{nl})}{b^{11}} \nabla_i h_{kl}\nabla_j h_{mn} \geq 0.  \tag{3.8}
\end{align*}
We conclude by  (\ref{eq:LB Pinching 1st}), (\ref{eq:LB pinching reaction}) and (\ref{eq:LB 3rd-IV}) that at $(p_0,t_0)$ the following holds
\begin{align*}
0\geq \frac{\cL w}{w} -\frac{\bd_t w}{w} \geq  -n\Big(1+\frac{1}{\alpha}\Big)(n\alpha -1)\frac{K^\alpha \grad^{-1}}{ \psi_\beta}+n\Big(1+\frac{1}{\alpha}\Big)\frac{\beta}{\psi_\beta}+K^{\alpha}\lambda_{\min}. 
\end{align*}
If  $n\alpha \leq 1$, then the last inequality gives $\displaystyle  n\Big(1+\frac{1}{\alpha}\Big) \leq 0$,  a contradiction. 
Hence  $t_0=0$, and therefore the desired result holds. 
If $n\alpha \geq 1$, then we use the improved inequality  (\ref{eq:LB pinching reaction second}) instead of  (\ref{eq:LB pinching reaction}) by 
and perform the same estimates  for all the other terms, so that at  $(p_0,t_0)$, we obtain 
\begin{align*}
0\geq \frac{\cL w}{w} -\frac{\bd_t w}{w} \geq  -n\Big(1+\frac{1}{\alpha}\Big)(n\alpha -1)\frac{K^\alpha \grad^{-1}}{ \psi_\beta}+n\Big(1+\frac{1}{\alpha}\Big)\frac{\beta}{\psi_\beta}+\frac{1}{n} K^{\alpha}H. 
\end{align*}
Hence 
\begin{align*}
n\Big(1+\frac{1}{\alpha}\Big)(n\alpha -1)\frac{\grad^{-1}}{ \psi_\beta}\geq n\Big(1+\frac{1}{\alpha}\Big)\frac{\beta}{\psi_\beta}K^{-\alpha}+\frac{1}{n} H. 
\end{align*}
Since $n\alpha \geq 1$, we have ${\displaystyle 1+\frac 1{\alpha} \leq 1+n}$,  and using also that  $\grad \geq 1$, $\psi_\beta \leq M$, and  $M \geq \beta$, we conclude from the 
previous inequality that    
\begin{align*}
n ^2(1+n)(n\alpha -1)\psi_\beta^{-1}\geq n^2\Big(1+\frac{1}{\alpha}\Big)K^{-\alpha}\frac{\beta}{\psi_\beta}+H \geq \Big(n^2\Big(1+\frac{1}{\alpha}\Big)K^{-\alpha}+H \Big)\frac{\beta}{M}
\end{align*}
Next, we employ the Young's inequality
\begin{align*}
 \frac{K^{-\alpha}}{(n-1)\alpha+1} +\frac{(n-1)\alpha H}{(n-1)\alpha+1} \geq K^{-\frac{\alpha}{(n-1)\alpha+1} }H^{\frac{(n-1)\alpha}{(n-1)\alpha+1}}=\big(K^{-1}H^{n-1}\big)^{\frac{\alpha}{(n-1)\alpha+1} } 
\end{align*}
and observe  the following
\begin{align*}
K^{-1}H^{n-1} = \frac{H^{n-1}}{\lambda_1 \lambda_2\cdots \lambda_n} = \frac{1}{\lambda_1}\, \frac{H^{n-1}}{\lambda_2\cdots \lambda_n} \geq \frac{1}{\lambda_1} =\lambda^{-1}_{\min}. 
\end{align*}
Combining the last three   inequalities  yields  
\begin{align*}
n^2(1+n)(n\alpha -1)M\beta^{-1}\psi_\beta^{-1} \geq  n^2\Big(1+\frac{1}{\alpha}\Big)K^{-\alpha}+ H \geq \frac{K^{-\alpha}}{(n-1)\alpha+1} +\frac{(n-1)\alpha H}{(n-1)\alpha+1} \geq \big(\lambda_{\min}^{-1}\big)^{\frac{\alpha}{(n-1)\alpha+1} }. 
\end{align*}
We conclude, that  if $n\alpha \geq 1$, the following holds at $(p_0,t_0)$
\begin{align*}
\lambda^{-1}_{\min} \, \psi_\beta^{n(1+\frac{1}{\alpha})} \leq \big(n^2(n+1)(n\alpha -1)M\beta^{-1}\big)^{(n-1)+\frac{1}{\alpha}}\, \psi_\beta^{1+\frac{n-1}{\alpha}}. 
\end{align*}
Thus, $\psi_\beta \leq M$ gives the desired result.
\end{proof}


\section{Speed estimates}\label{sec-speed} 
This section will be devoted to the proof of a  local speed bound. 
We recall the definition of the cut-off function   
$\psi(p,t)\coloneqq(M-\bar u(p,t))_+$,  where   $ \bar u(p,t) \coloneqq\langle F(p,t),\vec{e}_{n+1} \rangle$  denotes the height function.

\begin{theorem}[Local speed bound]  \label{thm:SE Speed Estimate}
Assume that  a smooth hypersurface $\Sigma_0$ satisfies the assumptions in \emph{Theorem \ref{thm:INT Existence}}  and let $\Sigma_t$ be a complete strictly convex smooth graph solution of \emph{(\ref{eq:INT aGCF})} defined on $M^n \times [0,T)$. Then, given a constant $M \geq 1$, 
\begin{equation*}
\Big( \frac{t}{1+t} \,\Big)(\psi^2 \,  K^{\frac{1}{n}}) (p,t) \leq (4n\alpha+1)^2(2\cW)^{1+\frac{1}{2n\alpha}}(\cW \cA +M^2)
\end{equation*}
where  $\cW$  and $\cA$ are constants given by 
\begin{align*}
\cW &= \sup \{ \grad^2(p,s) : \bar u(p,s) < M, \, s \in [0,t] \}, \\
\cA &= \sup \{ \lambda^{-1}_{\min}(p,s) :  \bar u(p,s) <  M,\,  s \in [0,t] \}.
\end{align*}
\end{theorem}

\begin{proof}
Choosing a fixed time $T_0 \in (0,T)$, we redefine $\cW$ and $\cA$ by
\begin{align*}
\cW = \sup \{ \grad^2(p,t) : \bar u(p,t) < M, \,   t \in [0,T_0] \}, \quad 
\cA = \sup \{ \lambda^{-1}_{\min}(p,t) : \bar u(p,t) <  M, \, t \in [0,T_0] \}. 
\end{align*}
Also, we define $\eta:[0,T)\to \mathbb{R}$ by
\begin{align*}
\eta(t)=\frac{t}{1+t}
\end{align*}
which will be used later in this proof.

Following the well used idea by Caffarelli, Nirenberg and Spruck in \cite{CNS88trick} (see also in  \cite{HE91MCFexist} and  \cite{CD15Qk}), we  define the  function $\varphi=\varphi(\grad^2)$, depending on $\grad^2$, by 
\begin{equation*}
\varphi (\grad^2) = \frac{\grad^2}{2\cW- \grad^2}.
\end{equation*}
The evolution equation of $\grad$ in  (\ref{eq:Pre grad_t}) gives 
\begin{align*}
\bd_t (\grad^2) =\cL(\grad^2)  - 2\alpha K^{\alpha}H\, \grad^2 - 6\|\nabla \grad\|^2_\cL.
\end{align*}
Then, the evolution equation of $ \varphi(\grad^2)$ is
\begin{align*} 
\bd_t \varphi  = \varphi' ( \cL\grad^2 - 2\alpha K^{\alpha}H\grad^2 - 6\|\nabla \grad\|^2_\cL) 
 = \cL \varphi -\varphi''\|\nabla \grad^2\|^2_\cL - \varphi' (2\alpha K^{\alpha}H \grad^2 + 6\|\nabla \grad\|^2_\cL).  
\end{align*}
Also, the evolution equation of $K^{\alpha}$ in  (\ref{eq:Pre aK_t}) leads to
\begin{align*}
\bd_t K^{2\alpha}   = \cL K^{2\alpha} -\frac{1}{2} K^{-2\alpha}\|\nabla K^{2\alpha}\|^2_\cL  + 2\alpha K^{3\alpha}H 
\end{align*}
implying the following   evolution equation for  $K^{2\alpha}\varphi(\grad^2)$ 
\begin{align*}
\bd_t (K^{2\alpha}\varphi) = &\cL (K^{2\alpha}\varphi)  -2\langle \nabla K^{2\alpha},\nabla \varphi \rangle_\cL + 2\alpha K^{3\alpha}H (\varphi - \varphi' \grad^2) \\
&-\frac{1}{2}\varphi K^{-2\alpha}\|\nabla K^{2\alpha}\|^2_\cL   - (4\varphi'' \grad^2 + 6\varphi')K^{2\alpha}\|\nabla \grad\|^2_\cL. 
\end{align*}
Observe that
\begin{align*}
-2\langle \nabla K^{2\alpha},\nabla \varphi \rangle_\cL 
 = &-\langle \nabla K^{2\alpha},\nabla \varphi \rangle_\cL  + \varphi^{-1}K^{2\alpha}\|\nabla \varphi \|^2_\cL - \varphi^{-1}\langle\nabla \varphi , \nabla 
 ( K^{2\alpha} \varphi) \rangle_\cL\\
 \leq &  \frac{1}{2}\varphi K^{-2\alpha}\|\nabla K^{2\alpha}\|^2_\cL  + \frac{3}{2}\varphi^{-1}K^{2\alpha}\|\nabla \varphi \|^2_\cL- \varphi^{-1}\langle\nabla \varphi , \nabla (K^{2\alpha}  \varphi) \rangle_\cL. 
\end{align*}
Hence, the following inequality holds
\begin{align*}\label{eq:SE f_t}
\bd_t (K^{2\alpha} \varphi)  \leq & \cL (K^{2\alpha}\varphi)  - \varphi^{-1}\langle\nabla \varphi , \nabla( K^{2\alpha} \varphi) \rangle_\cL +  2\alpha K^{3\alpha}H(\varphi - \varphi' \grad^2)\tag{4.1}
\\
& - (4\varphi'' \grad^2 + 6\varphi' -6\varphi^{-1} \varphi'^2\grad^2)\, K^{2\alpha}\|\nabla \grad\|^2_\cL. 
\end{align*}
Now,  we have  
$$\varphi(\grad^2)+1=\frac{2\cW}{2\cW-\grad^2}, \qquad \varphi'(\grad^2)=\frac{2\cW}{(2\cW-\grad^2)^2},  \qquad  \varphi''(\grad^2)=\frac{4\cW}{(2\cW-\grad^2)^3}.$$
Therefore, by direct calculation we  obtain
$$
\varphi - \varphi' \grad^2  =  \frac{\grad^2}{2\cW-\grad^2} -\frac{2\cW\grad^2}{(2\cW-\grad^2)^2}= -\frac{\grad^4}{(2\cW-\grad^2)^2} =-\varphi^2
$$
and 
$$ \varphi^{-1} \nabla \varphi  = \frac{2\cW-\grad^2}{\grad^2}\frac{4\cW\grad\nabla \grad}{(2\cW-\grad^2)^2}=4\cW\,  \varphi v^{-3} \nabla \grad $$
and also 
$$4\varphi'' \grad^2 + 6\varphi' -6\varphi^{-1} \varphi'^2\grad^2  = \frac{16\cW \grad^2}{(2\cW-\grad^2)^3}+\frac{12\cW }{(2\cW-\grad^2)^2}-\frac{24\cW^2}{(2\cW -\grad^2)^3} =\frac{4\cW }{(2\cW-\grad^2)^2}\, \varphi.
$$
Setting   $f\coloneqq K^{2\alpha}\, \varphi$ in  (\ref{eq:SE f_t})  and  using the identities  above  yields 
\begin{align*}
\bd_t  f &\leq  \cL f    -4\cW\, \varphi\grad^{-3} \langle\nabla \grad,\nabla f\rangle_\cL -2\alpha fK^{\alpha}H\varphi -\frac{4\cW}{(2\cW-\grad^2)^2}f\|\nabla \grad\|^2_\cL.
\end{align*}
On the other hand, (\ref{eq:Pre psi_t})  gives 
\begin{align*}
\bd_t \psi= \cL \psi  + (n\alpha-1)  \grad^{-1} K^\alpha \leq \cL \psi  + n\alpha  K^\alpha. 
\end{align*}
Hence, on the support of $\psi$, we have
\begin{align*}
 \bd_t \psi^{4n\alpha}\leq  \cL \psi^{4n\alpha} - 4n\alpha(4n\alpha-1)\psi^{4n\alpha-2}\|\nabla \psi\|^2_{\cL}+4 n^2\alpha^2 K^\alpha \psi^{4n\alpha-1}.
\end{align*}
Thus,   on the support of $\psi$, the following holds 
\begin{align*}
\bd_t (f\psi^{4n\alpha})   \leq & \cL (f\psi^{4n\alpha})  -2\langle\nabla \psi^{4n\alpha}, \nabla f\rangle_\cL  -4\cW \, \varphi\grad^{-3} \psi^{4n\alpha}\langle\nabla \grad,\nabla f\rangle_\cL
 -  2\alpha fK^{\alpha}H\varphi\psi^{4n\alpha}\\
 & -\frac{4\cW}{(2\cW-\grad^2)^2}f\psi^{4n\alpha}\|\nabla \grad\|^2_\cL -4n\alpha(4n\alpha-1)f\psi^{4n\alpha-2}\|\nabla \psi\|^2_\cL +4 n^2\alpha^2 K^\alpha \psi^{4n\alpha-1}.
\end{align*}
Next, we compute
\begin{align*}
 -4\cW\, \varphi\grad^{-3} &\psi^{4n\alpha}\langle\nabla \grad,\nabla f\rangle_\cL= \\
 = & -4\cW \, \varphi\grad^{-3} \langle\nabla \grad,\nabla (f \psi^{4n\alpha}) \rangle_\cL +16n\alpha\cW\, \varphi\grad^{-3} f\psi^{4n\alpha-1}\langle\nabla \grad,\nabla  \psi\rangle_\cL \\
\leq & -4\cW\, \varphi\grad^{-3} \langle \nabla \grad,\nabla (f \psi^{4n\alpha}) \rangle_\cL  +\frac{4\cW \, f\psi^{4n\alpha}\|\nabla \grad\|^2_\cL}{(2\cW-\grad^2)^2} +   16n^2\alpha^2\cW\, (2\cW-\grad^2)^2\varphi^2 \grad^{-6}f \psi^{4n\alpha-2}\|\nabla \psi\|^2_\cL \\
= & -4\cW \, \varphi\grad^{-3} \langle\nabla \grad,\nabla (f \psi^{4n\alpha}) \rangle_\cL  +\frac{4\cW }{(2\cW-\grad^2)^2} \,
f\psi^{4n\alpha}\|\nabla \grad\|^2_\cL+   16\, n^2\alpha^2\cW \grad^{-2}f \psi^{4n\alpha-2}\, \|\nabla \psi\|^2_\cL.
\end{align*}
Moreover, we have 
\begin{equation*}
-2\langle\nabla \psi^{4n\alpha}, \nabla f\rangle_\cL = -2\psi^{-4n\alpha} \langle\nabla \psi^{4n\alpha}, \nabla ( f \psi^{4n\alpha})  \rangle_\cL +32 n^2\alpha^2 \, \, f \psi^{4n\alpha-2}  \|\nabla \psi\|^2_\cL. 
\end{equation*}
Combining the above gives  
\begin{align*}
\bd_t (f\psi^{4n\alpha})  \leq  &\cL (f\psi^{4n\alpha}) 
-\langle 2\psi^{-4n\alpha} \nabla \psi^{4n\alpha} + 4\cW\varphi\grad^{-3} \nabla \grad,\nabla (f \psi^{4n\alpha}) \rangle_\cL -  2\alpha fK^{\alpha}H\varphi\psi^{4n\alpha}   \\
&  + \big (32\, n^2\alpha^2+ 16 \, n^2\alpha^2\cW \grad^{-2}-4n\alpha(4n\alpha-1) \big )f\psi^{4n\alpha-2}\|\nabla \psi\|^2_\cL +4 n^2\alpha^2 K^\alpha \psi^{4n\alpha-1}.
\end{align*}
In addition, on the support of $\psi$, we have $\nabla \psi= - \nabla \bar u = - \nabla \langle F ,\vec{e}_{n+1}\rangle$\, which leads to  
\begin{align*}
\|\nabla \psi\|^2_\cL & = \|\nabla \langle F , \vec{e}_{n+1} \rangle \|^2_\cL \leq \sum^{n+1}_{m=1}\|\nabla \langle F , \vec{e}_{m} \rangle \|^2_\cL = \sum^{n+1}_{m+1} \alpha K^{\alpha}b^{ij} \langle F_i,\vec{e}_{m} \rangle \, \langle F_j,\vec{e}_{m} \rangle \\
&= \sum^n_{i=1}\sum^n_{j=1} \alpha K^{\alpha}b^{ij} \Big( \sum^{n+1}_{m=1}  \langle F_i,\vec{e}_{m} \rangle \, \langle F_j,\vec{e}_{m} \rangle \Big) =  \sum^n_{i=1}\sum^n_{j=1}
\alpha K^{\alpha}b^{ij} g_{ij}  \leq n\alpha K^{\alpha}\lambda_{\min}^{-1}\leq n\alpha \cA K^{\alpha}.
\end{align*}
Hence, $\grad \geq 1$ implies 
\begin{equation*}
\big (32n^2\alpha^2+16 n^2\alpha^2\cW\,  \grad^{-2}-4n\alpha(4n\alpha-1) \big )\, f\psi^{4n\alpha-2}\, \|\nabla \psi\|^2_\cL \leq n\alpha \, \big (16n^2\alpha^2(\cW+1)+4n\alpha \big )f\psi^{4n\alpha-2}\cA K^{\alpha}.
\end{equation*}
Thus, by the inequalities $H \geq nK^{\frac{1}{n}}$ and $\varphi \geq {1}/{(2\cW)}$, the evolution equation of $f\psi^{4n\alpha}$ can be reduced to the following 
\begin{align*}
\bd_t (f\psi^{4n\alpha} ) \leq  & \cL (f\psi^{4n\alpha})  -\langle 2\psi^{-4n\alpha} \nabla \psi^{4n\alpha}+ 4\cW\varphi\grad^{-3} \nabla \grad,\nabla f \psi^{4n\alpha}\rangle_\cL
-  n\alpha \, \cW^{-1}  K^{\alpha +\frac{1}{n}}\, f \psi^{4n\alpha}\\ 
& +4n^2\alpha^2 (4n\alpha(\cW +1)+1)\cA K^{\alpha} \, f\psi^{4n\alpha-2}+4 n^2\alpha^2 K^\alpha \psi^{4n\alpha-1}.
\end{align*}
Involving $\eta\eqcolon t(1+t)^{-1}$ and $\bd \eta_t = (1+t)^{-2}\leq 1$ yields
\begin{align*}
\bd_t (\eta^{2n\alpha} f\psi^{4n\alpha} ) \leq  & \cL (\eta^{2n\alpha}f\psi^{4n\alpha})  -\langle 2\psi^{-4n\alpha} \nabla \psi^{4n\alpha}+ 4\cW\varphi\grad^{-3} \nabla \grad,\nabla \eta^{2n\alpha} f \psi^{4n\alpha}\rangle_\cL
-  n\alpha \, \cW^{-1}  K^{\alpha +\frac{1}{n}}\eta^{2n\alpha} f \psi^{4n\alpha}\\ 
&  +4n^2\alpha^2 (4n\alpha(\cW +1)+1)\cA K^{\alpha} \eta^{2n\alpha} f\psi^{4n\alpha-2}+4 n^2\alpha^2 K^\alpha  \eta^{2n\alpha} f \psi^{4n\alpha-1} +2n\alpha\eta^{2n\alpha-1} f\psi^{4n\alpha}.  
\end{align*}
Now by  conditions  (ii), (iii) in Theorem \ref{thm:INT Existence}   $\psi$ is compactly supported. Hence $\eta^{2n\alpha}f\psi^{4n\alpha}$ attains its maximum in $M^n \times [0,T_0] $ at some $(p_0,t_0)$ with $t_0>0$. Then,  the last  inequality implies that at $(p_0,t_0)$  
\begin{align*}
n\alpha \, \cW^{-1}  K^{\alpha +\frac{1}{n}}\eta^{2n\alpha} f \psi^{4n\alpha}\leq  \, &  4n^2\alpha^2 (4n\alpha(\cW +1)+1)\cA K^{\alpha} \eta^{2n\alpha} f\psi^{4n\alpha-2}\\
&+4 n^2\alpha^2 K^\alpha  \eta^{2n\alpha} f \psi^{4n\alpha-1} +2n\alpha\eta^{2n\alpha-1} f\psi^{4n\alpha}.
\end{align*}
Multiplying  by $(n\alpha)^{-1}\cW K^{-\alpha} \eta^{-2n\alpha+1} f^{-1} \psi^{-4n\alpha +2}$ yields the bound 
\begin{align*}
 \eta K^{\frac{1}{n}}\psi^2\leq  \,   4n\alpha \cW(4n\alpha(\cW +1)+1)\cA \, \eta +4 n\alpha\cW  \eta  \psi +2\cW K^{-\alpha} \psi^2
\end{align*}
and by  $\cW \geq 1 $, $\psi \leq M$, $1\leq M$, and $\eta \leq 1$
\begin{align*}
\eta K^{\frac{1}{n}}\psi^2\leq \, &  4n\alpha  \cW \,\eta \Big (   \big(4n\alpha(\cW + \cW)+\cW \big) \, \cA + \psi \Big )+ 2\cW \eta^{n\alpha}\psi^{2+2n\alpha}(\eta K^{\frac{1}{n}}\psi^2)^{-n\alpha}\\
 \leq  \, &  4n\alpha \cW (8n\alpha +1)  (\cW\cA + M)+2\cW M^{2+2n\alpha}(\eta K^{\frac{1}{n}}\psi^2)^{-n\alpha} \\
  \leq \, &  2 \cW (16n^2\alpha^2 +2n\alpha+M^{2n\alpha}(\eta K^{\frac{1}{n}}\psi^2)^{-n\alpha})  (\cW\cA + M^2). 
\end{align*}
Hence, in the case of $\eta K^{\frac{1}{n}}\psi^2 \geq M^2$, the last inequality yields
\begin{align*}
\eta K^{\frac{1}{n}}\psi^2\leq   2 \cW (16n^2\alpha^2 +2n\alpha+1)  (\cW\cA + M^2)\leq 2 \cW (4n\alpha+1)^2  (\cW\cA + M^2).
\end{align*}
In the other case, we can simply obtain \ $\eta K^{\frac{1}{n}}\psi^2 \leq M^2\leq 2 \cW (4n\alpha+1)^2  (\cW\cA + M^2)$. Thus, at $(p_0,t_0)$,
\begin{align*}
\eta K^{\frac{1}{n}}\psi^2\leq 2 \cW (4n\alpha+1)^2  (\cW\cA + M^2).
\end{align*}
Let $\Psi$ denote the maximum value  $\eta^{2n\alpha}f\psi^{4n\alpha}(p_0,t_0)= \eta^{2n\alpha}\varphi K^{2\alpha}\psi^{4n\alpha}(p_0,t_0)$. Then, $\varphi \leq 1$ gives 
\begin{align*}
\Psi \leq  (\eta K^{\frac{1}{n}}\psi^2)^{2\alpha n}(p_0,t_0)  \leq (2\cW)^{2\alpha n}( 4n\alpha+1)^{4\alpha n} (\cW\cA +  M^2)^{2\alpha n}.
\end{align*}
Using also that    $(2\cW)^{-1}\leq \varphi$,  we finally conclude that for all $p \in M^n$  and $t \in [0,T_0]$ the following holds 
\begin{align*}
\frac{\eta^{2n\alpha} K^{2\alpha} \psi^{4n\alpha}(p,t)}{2\cW}\leq  \eta^{2n\alpha}\varphi \, K^{2\alpha} \psi^{4n\alpha} (p,t)\leq  \Psi  \leq (2\cW)^{ 2\alpha n}( 4n\alpha+1)^{4\alpha n} (\cW\cA +  M)^{2\alpha n}.
\end{align*}
Hence, setting $t =T_0$ yields 
\begin{align*}
(\eta K^{\frac{1}{n}}\psi^2)  (p,T_0) \leq (2\cW)^{1+\frac{1}{2n\alpha}}(4n\alpha+1)^2(\cW \cA +M^2)
\end{align*}
and the  desired result simply follows by substituting $T_0$ by  $t$. 
\end{proof}

\section{Long time existence}\label{sec-let}

In this final section, we will establish the all time 
existence of the complete non-compact $\alpha$-Gauss curvature flow (\ref{eq:INT aGCF}), as stated in our main Theorem \ref{thm:INT Existence}.
Our proof will be based on the a'priori estimates in sections  \ref{sec-prelim}-\ref{sec-speed}, and the proof will be done in two steps. We will first show, in the next Theorem,  the existence of a complete solution $\Sigma_t$ on $t \in (0,T)$,  where  $T=T_\Omega$ depends on the domain 
$\Omega$.   We will then construct an 
appropriate barrier to guarantee that each $\Sigma_t$ remains as a graph over the same domain $\Omega$,  for all $t \in (0,T)$, 
implying  that $T=+\infty$ independently from the domain $\Omega$.

\begin{theorem}\label{thm:LTE Long time existence} Let $\Omega$, $ u_0$  and $\Sigma_0$ satisfy the conditions in \emph{Theorem \ref{thm:INT Existence}}. Assume $B_R(x_0) \subset \Omega$, for some $R>0$. Then, given an immersion 
$F_0:M^n \rightarrow \rno$ with $F_0(M^n)=\Sigma_0$, there is a solution $F:M^n \times (0,T) \rightarrow \rno$ of \emph{(\ref{eq:INT aGCF})} for some $T \geq (n\alpha+1)R^{n\alpha+1}$ such that for each $t \in (0,T)$, the image $\Sigma_t \coloneqq F(M^n,t)$ is a strictly convex smooth complete graph of a function $u(\cdot,t):\Omega_t \to \mathbb{R}$ defined on a convex open $\Omega_t \subset \Omega$, and also $u(\cdot,t)$ and $\Omega_t$ satisfy the conditions  of $u_0$ and $\Omega$ in \emph{Theorem \ref{thm:INT Existence}}.
\end{theorem}

We begin with some  extra notation.

\begin{notation}\label{not2}  We have:  

\begin{enumerate}
\item  Given a set $A \in \rno$, we denote by $\cv(V)$ its \emph{convex hull} $\{ tx+(1-t)y:  x,y \in A, t \in [0,1]\}$.
\item  Let $\Sigma$ be a convex complete (either non-compact or closed) hypersurface.  If a set $V$ is a subset of  $\cv(\Sigma)$, we say $V$ is \emph{enclosed} by $\Sigma$ and use the notation  
\begin{equation*}
V \preceq \Sigma. 
\end{equation*}
In particular, if $V \bigcap \Sigma =\emptyset$ and $V \preceq \Sigma $, we use $V \prec \Sigma$.

\item For a convex smooth hypersurface $\Sigma$ with a point $X\in \Sigma$, we denote by $K(\Sigma)(X)$ the Gauss curvature of $\Sigma$ at $X$.

\item For a convex complete graph $\Sigma$ with a point $X\in \Sigma$, we define $\bar u (\Sigma)(X)$ and $\grad_{\min}(\Sigma)(X)$ by
\begin{align*}
\bar u (\Sigma)(X)= \langle X,\vec{e}_{n+1}\rangle ,\quad \grad(\Sigma)(X)=\sup\big \{\langle \vec{n}_L,\vec{e}_{n+1} \rangle^{-1}: L \; \text{is a hyperpalne tangent to} \; \Sigma\; \text{at}\; X  \,\big\}
\end{align*}
where $\vec{n}_L$ is the upward unit normal vector of a hyperplane $L$.

\item $B^{n+1}_R(Y)\coloneqq \{X \in \rno: |X-Y|<R \}$ 
denotes the \emph{$(n+1)$-ball of  radius $R$ centered at $Y \in\rno$}. 

\item For a convex closed hypersurface $\Sigma$, we define the \textit{support function} $S:S^n\to \mathbb{R}$ by
\begin{align*}
S(v)=\max_{Y\in \Sigma}\,\langle v,Y\rangle
\end{align*}

\item For a convex hypsersurface $\Sigma$ and $\eta>0$, we denote by $\Sigma^\eta$ the $\eta$-envelope of $\Sigma$.
\begin{align*}
\Sigma^{\eta} = \big \{Y \in \rno :d(Y,\Sigma)=\eta,\, Y \not \in \cv(\Sigma) \big\}
\end{align*}
where $d$ is the distance function.

\item  For $r >0$ and $(x_0,t_0) \in \rn \times (0,+\infty)$, 
$Q_r((x_0,t_0))\coloneqq B_r(x_0) \times (t_0 - r^2, t_0]$ denotes  the parabolic cube centered at $(x_0,t_0)$, where $B_r(x_0)$ 
is the  $n$-ball of  radius $r$ centered at $x_0 \in\rn$. Also, for $\beta \in (0,1)$,   
$C^{\beta,\beta/2}_{x,t}(Q_r)$ denotes  the standard H\"older space with respect to the parabolic distance. 
\end{enumerate}
\end{notation}

In the proof of Theorem \ref{thm:LTE Long time existence}
we  will use a standard Schauder  estimate for  equation \eqref{eq:INT aGCFeq}. However, since  \eqref{eq:INT aGCFeq} is not a concave equation we cannot directly use the known  regularity theory. To detour  this difficulty, we slightly modify the standard  $C^{2,\beta}$ estimate of G. Tian and X.-J. Wang  in  \cite{W12Schauder}.  

\begin{proposition}[$C^{2,\beta}$ estimate]\label{prop:LTE C^(2,beta) Estimates}
Let $ u: Q_r  \rightarrow \mathbb{R}$ be a strictly convex and smooth solution of
 \emph{(\ref{eq:INT aGCFeq})} with $Q_r\coloneqq Q_r((x_0,t_0))$.  
Then, there exist some $\beta \in (0,1)$ such that for any $\sigma \in (0,1)$, we have 
\begin{equation*}
\| D^2 u \|_{C^{\beta,\beta/2}_{x,t}(Q_{\sigma  r})} \leq C(r, \sigma ,\alpha, \beta , n, \; \sup_{Q_r} |u| , \; \sup_{Q_r}\grad, \;\sup_{Q_r}K, \; \inf_{Q_r}\lambda_{\min}) 
\end{equation*}
where $\grad$, $K$, and $\lambda_{\min}$ are the gradient function, the Gauss curvature, and the smallest principal curvature of $\Sigma_{t}=\{(y, u(y,t)): y \in B_r(x_0) \}$ at $(x, u(x,t))$.
\end{proposition}

\begin{proof}
One can easily show that $\displaystyle  \sup_{Q_r}\grad, \;\sup_{Q_r}K, \; \inf_{Q_r}\lambda_{\min}$ control the ellipticity constant from above and below of the fully-nonlinear operator in \eqref{eq:INT aGCFeq}. Thus, space-time H\"{o}lder estimate for $\bd_t u$ can be obtained by Krylov-Safonov's estimate in \cite{KS81Holder}, as in Step 1 of the proof of Theorem 2.1 in \cite{W12Schauder}. Now, we transform \eqref{eq:INT aGCFeq} into the following equation
\begin{align*}
(\det D^2 u)^{\frac{1}{n}} = (u_t)^{\frac{1}{n\alpha}}(1+|D u|^2)^{\frac{(n+2)\alpha-1}{2n\alpha }}. 
\end{align*}
Since $(\det D^2 u)^{\frac{1}{n}}$ is a concave operator, the result of Caffarelli in \cite{CC95EPDE} gives us a space H\"{o}lder estimate for $D^2 u(\cdot,t)$ for each $t$. Thus,  H\"{o}lder estimate for $D^2u$ in $t$ can be achieved in the same manner as in Step 2 of the proof of Theorem 2.1 in \cite{W12Schauder}.
\end{proof}

We will now give the proof of the existence Theorem \ref{thm:LTE Long time existence}. 

\begin{proof}[Proof of Theorem \ref{thm:LTE Long time existence}] 
We will obtain a solution $\Sigma_t\coloneqq \{ (x,    u(\cdot,t)):  x \in \Omega_t\subset \rn\}$
as a limit  $$  \Sigma_t\coloneqq \lim_{j \to +\infty}   \Sj^j_t$$   where  $\Sj_t^j$ is 
 a strictly convex closed hypersurface which is symmetric with respect to the hyperplane $x_{n+1}=j$
and also evolves by the $\alpha$-Gauss curvature flow (\ref{eq:INT aGCF}). Let $\Sigma_t^j \coloneqq \Sj^j_t \bigcap \big( \rn \times [0,j] \big)$
denote the lower half of $\Sj^j_t$. 
Then, the symmetry guarantees that each   $\Sigma_t^j$  is a graph over the same hyperplane $\rn$. Thus, by applying the local a'priori estimates shown in sections
\ref{sec-prelim}-\ref{sec-speed} on compact subsets of $\rno$, we obtain 
uniform $C^{\infty}$ bounds on the lower half of $\Sigma_t^j$ necessary to
pass to the limit. 

\smallskip

\,  {\em Step 1 : The construction of the approximating sequence $\Sigma_t^j$.} \,   Let  $u_0, \Sigma_0$ and $\Omega$ be as in  
Theorem \ref{thm:INT Existence}  and assume that $\displaystyle \inf_{\Omega} u_0 = 0$.  

For each $j \in \mathbb{N}$,  we  reflect 
$\Sigma_0 \bigcap \big( \rn \times [0,j]\big)$ over the $j$-level hyperplane $\rn \times \{j\}\coloneqq\{(x,j):x \in \rn\}$ to  obtain a uniformly convex closed hypersurface $\bar \Sj^j_0$ defined by
\begin{align*}
\bar \Sj^j_0 = \{(x,h) \in \rno :  h \in \{u_0(x),2j-u_0(x)\},x \in \Omega, u_0(x) \leq j  \}. 
\end{align*}
Then, we let $\Sj^j_0$ denote  the $(1/j)$-envelope of $\bar \Sj^j_0$, which is uniformly convex closed hypersurface of class $C^{1,1}$ and we denote 
it by   $(\bar \Sj^j_0)^{1/j}$. Now, by Theorem 15 in \cite{A00aGCF}, there is a unique convex closed viscosity solution $\Sj^j_t$ of (\ref{eq:INT aGCF})
with initial data  $\bar \Sj^j_0$ which is defined for $t \in (0,T_j)$, where $T_j$ is its  maximal existing time. 
In addition, the uniqueness guarantees the symmetry of $\Sj^j_t$ with respect to the hyperplane  $\rn \times \{j\}$. Hence, its lower half  $\Sigma_t^j \coloneqq \Sj^j_t \bigcap \big( \rn \times [0,j] \big)$ is given  by  the  graph of a function $u^j(\cdot,t)$ defined on a convex set  $\Omega^j_t \subset \rn$, namely 
$$\Sigma^j_t  \coloneqq \Sj^j_t \bigcap \big( \rn \times [0,j] \big) = \{(x,u^j(x,t)):x \in \Omega^j_t\}.$$
Finally, we set
$$\Sigma_t = \bd \big \{ \bigcup_{j\in \mathbb{N}}\cv(\Gamma^j_t)\big \}, \quad \Omega_t = \bigcup_{j \in \mathbb{N}} \Omega^j_t, \quad t \in [0,T), \quad  \mbox{where} \,\,
 T=\sup_{j \in \mathbb{N}} T_j.
$$

\medskip

\noindent{\em Step 2 : Properties of  the approximating sequence $\Sigma_t^j$.} \, The following hold: 
\begin{enumerate}[label=(\subscript{P}{\arabic*})]
\item $\Sj^j_t \subset \Sj^{j+1}_t $, $T_j \leq T_{j+1}$, and $u_0(y)\leq  u^{j+1}(y,t) \leq  u^{j}(y,t) $ hold for all $j\in \mathbb{N}$, $t < T_j$, and $y \in \Omega_{t}^j $. 
\item If $B_R(x_0) \subset \Omega$, then $T \geq (n\alpha+1)^{-1}R^{n\alpha+1}$.
\end{enumerate}

\noindent\textit{Proof of properties} ($P_1$) - ($P_2$).
Since we have $ \Sj_0^{j} \preceq  \Sj_0^{j+1}$, the  comparison principle gives that  $ \Sj_t^{j}\preceq \Sj_t^{j+1}$, which yields  $u_0(y)\leq  u^{j+1}(y,t) \leq  u^{j}(y,t) $.
Moreover, since $\Sj^j_t$ exists until it converges to a point by Theorem 15 in \cite{A00aGCF}, we also have $T_j \leq T_{j+1}$.

Let us next show   (\textit{$P_2$}). $B_R(x_0) \subset \Omega$ means that there is a constant $h_0$ satisfying $B^{n+1}_R((x_0,h_0))\preceq \Sigma_0$. Set  $X_0=(x_0,h_0)$ and choosing  $j \geq R+h_0$ so that 
$B^{n+1}_R(X_0)\preceq \bar \Sj^j_0$ holds. On the other hand,  $ \bd B^{n+1}_{\rho_j(t)}(X_0)$ is a solution of (\ref{eq:INT aGCF}), where  $ \rho_j(t) = (R^{n\alpha +1}-(n\alpha+1)t)^{\frac{1}{n\alpha+1}}$. Hence, the comparison principle leads to $\bd B^{n+1}_{\rho_j(t)}(X_0) \preceq \Sj^j_t$, and by Theorem 15 in \cite{A00aGCF}, $\Sj^j_t$ exists while $\rho_j(t)>0$. Thus, 
($P_2$) holds.

\bigskip

\noindent{\em Step 3 : A'priori estimates for $\Sj^j_t$.} We verify that the a'priori estimates in section \ref{sec-prelim}-\ref{sec-speed} apply $\Sj^j_t$ for the cut-off $\psi_\beta$ with $  M < j$, even though $\Sj^j_0$ fails to be smooth.

By considering $(0,j)$ as a origin, we have a positive and symmetric support function $S^j_0$ of $\Sj^j_0$. We recall the compactly supported mollifiers $\varphi_\epsilon$ on $S^n$ in Proposition 6.2 in \cite{CD15Qk}, and let $S^{j,\epsilon}_0$ denote the convolution $S^j_0*\varphi_\epsilon$ on $S^n$. Then, by Proposition 6.3 in \cite{CD15Qk}, for a small $\epsilon$, $S^{j,\epsilon}_0$ is the support function of a strictly convex smooth closed hypersurface $\Sj^{j,\epsilon}_0$. Let $\Sj^{j,\epsilon}_t$ be the unique strictly convex smooth closed solution of \eqref{eq:INT aGCF} defined for $t \in [0, T^{j,\epsilon}_0)$ (c.f. in \cite{Chow85GCF}). Moreover, the uniqueness and the symmetry of  $\Sj^{j,\epsilon}_0$ guarantees that $\Sj^{j,\epsilon}_t$ also has the symmetry with respect to $\rn \times \{j\}$. Therefore, the lower half of $\Sj^{j,\epsilon}_t$ remains as a graph and thus a'priori estimates in sections \ref{sec-prelim}-\ref{sec-speed} apply for $\Sj^{j,\epsilon}_t$ with $M \leq j$.

On the other hand, by Proposition 6.3 in \cite{CD15Qk}, $\Sj^{j,\epsilon}_0$ converge to $\Sj^j_0$ in $C^1$ sense and the following holds
\begin{align*}
\liminf_{\epsilon \to 0}\lambda_{\min}(\Sj^{j,\epsilon}_0)(X_\epsilon) \geq \lambda^{\text{loc}}_{\min}(\Sj^j_0)(X)
\end{align*}
where $\{X_\epsilon\}$ is a set of points $X_\epsilon \in \Sj^{j,\epsilon}_0$ converging to $X\in \Sigma^j_0$ as $\epsilon \to 0$. 

Thus, for $M < j$ and $Q_M \eqqcolon \rn \times [-1,M)$, the following hold 
\begin{align*}
&\liminf_{\epsilon \to 0}\big( \inf_{X\in Q_M}\lambda_{\min}(\Sj^j_0)(X)\big) \geq \inf_{X\in Q_M}\lambda^{\text{loc}}_{\min}(\Sj^j_0)(X), && \limsup_{\epsilon \to 0}\big( \sup_{X\in Q_M}\grad(\Sj^j_0)(X)\big) \leq \sup_{X\in Q_M}\grad(\Sj^j_0)(X)
\end{align*}

Hence, Theorem \ref{thm:Pre Gradient} and Theorem \ref{thm:LB Local Lower bounds} gives the uniform gradient estimate and the uniform lower bounds of principal curvatures for $\Sj^{j,\epsilon}_t$, and therefore Theorem \ref{thm:SE Speed Estimate} guarantees the uniform upper bound for principal curvatures for $\Sj^{j,\epsilon}_t$. Now, we let the lower half of $\Sj^{j,\epsilon}_t$ be the graph of a function $u^{j,\epsilon}(0,t)$. Then, we have $|u^{j,\epsilon}|\leq j$ and thus Proposition \ref{prop:LTE C^(2,beta) Estimates} implies the local uniform $C^{2,\beta}$ estimates for $u^{j,\epsilon}$. Hence, the limit $\displaystyle \lim_{\epsilon \to 0}u^{j,\epsilon}= u^j$ satisfies the same estimates, and the lower half of $\Sj^j_t$ is a smooth graph for $t>0$.

\medskip

{\em Step 4 : Passing $\Sj^j_t$ to the limit $\Sigma_t$.}  First we observe that $ \bigcup_{j\in \mathbb{N}}\cv(\Gamma^j_t)$ is a convex body by ($P_1$); $\Gamma^j_t \preceq \Gamma^{j+1}_t \preceq \Sigma_0$. Thus, $\Sigma_t$ is a complete and convex hypersurface embedded in $\rno$. 

For any constant $M>0$, $t_0 >0$, $\sigma >0$ , and $j > M$, we apply the gradient estimate Theorem \ref{thm:Pre Gradient}  with $\beta=\min\{M,t_0^{-1}\sigma \}$, which yields a uniform gradient bound of $\Sj^j_t$ for $j >M$ and $t \in [0,t_0]$ in $\rn \times [0,M-2\sigma]$. So, the gradient function $\grad(\Sigma_t)$ of $\Sigma_t$ is bounded in each $\rn \times [0,M-2\sigma]$. Thus,
\begin{center}
\textit{$\Sigma_t$ is a complete convex graph of a function $u(\cdot,t)$ defined on a convex open set  $\Omega_t$.}
\end{center} 
In addition, by ($P_1$); $u_0(x) \leq u(x,t)$, 
\begin{center}
\textit{$u(\cdot,t)$ and $\Omega_t$ satisfy the conditions {\em(i),(ii),(iii)} of $u_0$ and $\Omega$ in {\em Theorem \ref{thm:INT Existence}}.}
\end{center} 

Also, by applying the curvature lower bound, Theorem \ref{thm:LB Local Lower bounds}, and the speed bound, Theorem \ref{thm:SE Speed Estimate}, we have a uniform curvature bound from below and above of $\Sj^j_t$ in $[0,M-\sigma]$ for large $j>M$, small $\sigma>0$, and $t \in [t_1,t_0]$ with $0 < t_1 < t_0< T $.

\medskip

\noindent {\em Step 5 : Passing $u^j$ to the limit $u$.}  Recall $ u_j$ the sequence of solutions to \eqref{eq:INT aGCFeq} defined in Step 1. For  $t \in (0,T)$, by (\textit{$P_1$}), the following holds 
\begin{align*}
 u(y,t) = \lim_{ j\rightarrow \infty}  u^j(y,t).
\end{align*}

We begin by choosing  $t_1,t_0 \in (0,T)$, ($t_1<t_0$) and $y_0 \in \Omega_{t_0}$. Since $\Omega_{t_0}$ is open, there is a small $r>0$ satisfying $\overline{B_{r}(y_0)} \subset \Omega_{t_0}$ and $r^2 < t_0-t_1$. 
On the other hand, (\textit{$P_1$}) gives  the monotone convergence of $\Omega^j_{t_0}$ to $\Omega_{t_0}$. Hence, for some large $J_0\in \mathbb{N}$, we have $B_{r}(y_0)\subset \Omega^{J_0}_{t_0}$, implying that  $ u^{J_0}(y,t_0) \leq J_0$,  for all $y \in B_r(y_0)$. 
Thus, for all $j \geq J_0$, $t \in [t_1,t_0]$, and $y \in B_r(y_0)$, the convexity of $ u^j$ and (\textit{$P_1$}) lead to
\begin{equation*}\label{eq:LTE Domain}
0 \leq u_0(y) \leq  u^j(y,t) \leq  u^j(y,t_0) \leq  u^{J_0}(y,t_0) \leq J_0. \tag{6.1}
\end{equation*}
Notice that in Step 4,  we have shown uniform estimates for $\lambda_{\min}^{-1}$, $\grad$, and $K$ of $\{\Sigma^j_t\}_{t \in [t_1,t_0]}$ in $\rn \times [0,J_0]$ for  $j >J_0$.
Therefore, Proposition \ref{prop:LTE C^(2,beta) Estimates} and (\ref{eq:LTE Domain}) imply  that $ u$ is of class $C_{\text{loc}}^{2,\beta}(B_r(y_0) \times (t_1,t_0])$ for some $\beta \in (0,1)$, and $ u(y,t)$ is a strictly convex $C^{2,\beta}$ solution of (\ref{eq:INT aGCFeq}). Hence, standard regularity results show that actually $ u(y,t)$ is $C^\infty$ smooth. Thus, by (\textit{$P_2$}), we assure that $\Sigma_t$ is the desired solution. This finishes the proof of Theorem \ref{thm:LTE Long time existence}. 
\end{proof}

\bigskip 
\bigskip 

To finish with the proof of  Theorem \ref{thm:INT Existence}, it remains  to  show that  $$\Omega_t =\Omega$$ which follows from the next Theorem.

\begin{theorem}
Let $\Omega$, $ u_0$  and $\Sigma_0$ satisfy the conditions in \emph{Theorem \ref{thm:INT Existence}}. Assume $\Sigma_t=\{(x,u(x,t)):\Omega_t\}$, $t \in [0,T)$, is a strictly convex smooth complete graph solution of \emph{(\ref{eq:INT aGCF})} such that $\Omega_t$, $u(\cdot,t)$ and $\Sigma_t$ satisfy the conditions of $\Omega$, $u_0$ and $\Sigma_0$ in \emph{Theorem \ref{thm:INT Existence}}. Then, for any closed ball $\overline{B_{R_0}(y_0)} \subset \Omega$ and any $t_0 \in (0,T)$, the following holds
\begin{equation*}
B_{R_0}(y_0)\subset \Omega_{t_0}. 
\end{equation*}
\end{theorem}

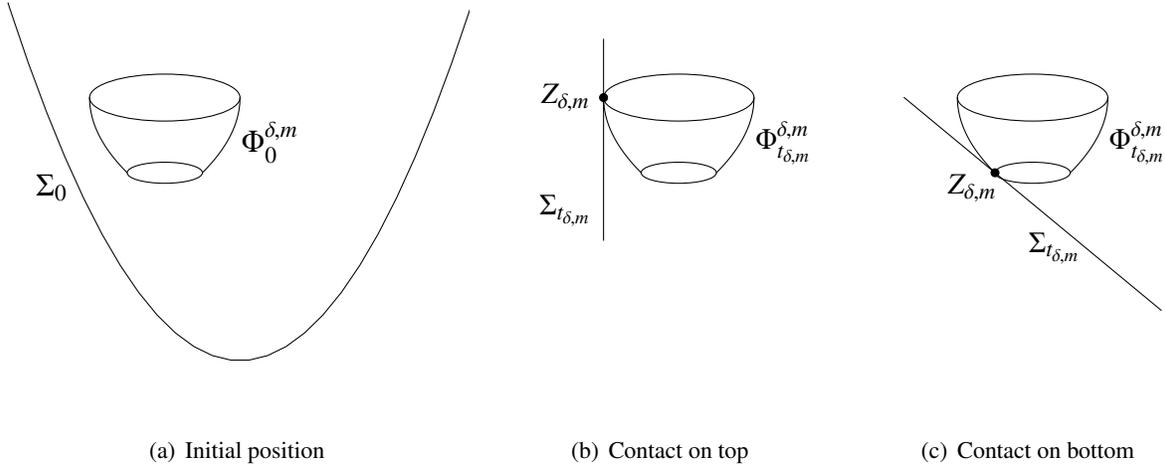
\begin{figure}[h]\label{figure-super}
\subfigure[Initial position]{
\begin{tikzpicture}[line cap=round,line join=round,>=triangle 45,x=1.0cm,y=1.0cm]
\clip(-0.62688201647,-4.24574436163) rectangle (5.54442150691,1.25790351741);
\draw [samples=50,rotate around={-90.:(0.5,0.)},xshift=0.5cm,yshift=0.cm,domain=0:1)] plot (\x,{(\x)^2/2/1.0});
\draw [samples=50,rotate around={-270.:(2.5,0.)},xshift=2.5cm,yshift=0.cm,domain=-1:0)] plot (\x,{(\x)^2/2/1.0});
\draw [rotate around={0.:(1.5,0.)}] (1.5,0.) ellipse (1.cm and 0.31224989992cm);
\draw [rotate around={0.:(1.5,-1.)}] (1.5,-1.) ellipse (0.5cm and 0.14cm);
\draw [samples=50,rotate around={0.:(2.5,-3.5)},xshift=2.5cm,yshift=-3.5cm,domain=-6.0:6.0)] plot (\x,{(\x)^2/2/1.0});
\draw[color=black] (2.9,-0.6) node[scale=1.1] {$\Phi^{\delta,m}_{0}$};
\draw[color=black] (0,-1.2) node[scale=1.1] {$\Sigma_0$};
\end{tikzpicture}
}
\subfigure[Contact on top]{
\begin{tikzpicture}[line cap=round,line join=round,>=triangle 45,x=1.0cm,y=1.0cm]
\clip(-1,-4.24574436163) rectangle (3.5,1.25790351741);
\draw [samples=50,rotate around={-90.:(0.5,0.)},xshift=0.5cm,yshift=0.cm,domain=0:1)] plot (\x,{(\x)^2/2/1.0});
\draw [samples=50,rotate around={-270.:(2.5,0.)},xshift=2.5cm,yshift=0.cm,domain=-1:0)] plot (\x,{(\x)^2/2/1.0});
\draw [rotate around={0.:(1.5,0.)}] (1.5,0.) ellipse (1.cm and 0.31224989992cm);
\draw [rotate around={0.:(1.5,-1.)}] (1.5,-1.) ellipse (0.5cm and 0.14cm);
\draw (0.5,-1.90012982601) -- (0.5,0.777364479141);
\begin{scriptsize}
\draw [fill=black] (0.5,0.) circle (1.5pt);
\draw[color=black] (2.9,-0.6) node[scale=1.4] {$\Phi^{\delta,m}_{t_{\delta,m}}$};
\draw[color=black] (0,0) node[scale=1.4] {$Z_{\delta,m}$};
\draw[color=black] (0,-1.5) node[scale=1.4] {$\Sigma_{t_{\delta,m}}$};
\end{scriptsize}
\end{tikzpicture}
}\quad
\subfigure[Contact on bottom]{
\begin{tikzpicture}[line cap=round,line join=round,>=triangle 45,x=1.0cm,y=1.0cm]
\clip(-0.62688201647,-4.24574436163) rectangle (3.5,1.25790351741);
\draw [samples=50,rotate around={-90.:(0.5,0.)},xshift=0.5cm,yshift=0.cm,domain=0:1)] plot (\x,{(\x)^2/2/1.0});
\draw [samples=50,rotate around={-270.:(2.5,0.)},xshift=2.5cm,yshift=0.cm,domain=-1:0)] plot (\x,{(\x)^2/2/1.0});
\draw [rotate around={0.:(1.5,0.)}] (1.5,0.) ellipse (1.cm and 0.31224989992cm);
\draw [rotate around={0.:(1.5,-1.)}] (1.5,-1.) ellipse (0.5cm and 0.14cm);
\draw [domain=-0.209605657245:3.21017674721] plot(\x,{(-0.0857864376269-0.414213562373*\x)/0.5});
\begin{scriptsize}
\draw [fill=black] (1.,-1.) circle (1.5pt);
\draw[color=black] (0.7,-1.2) node[scale=1.4] {$Z_{\delta,m}$};
\draw[color=black] (1.8,-2) node[scale=1.4] {$\Sigma_{t_{\delta,m}}$};
\draw[color=black] (2.9,-0.6) node[scale=1.4] {$\Phi^{\delta,m}_{t_{\delta,m}}$};
\end{scriptsize}
\end{tikzpicture}
}
\caption{Properties of  barrier }
\end{figure}

\begin{proof}
Let $y_0 \in \Omega$, $R_0$ and $t_0 \in (0,T)$ as in the statement of the Theorem.  Without loss of generality, we may assume  that $y_0=0$ and $R_0<1$. The given condition $\overline{B_{R_0}(0)} \subset \Omega$ implies that  there exists a constant $m_0 \geq 0$ such that $\overline{B_{R_0}(0)} \prec L_{m_0}(\Sigma_0)$. For  each $m \geq m_0+1$ and each $\delta>0$ sufficiently small  satisfying
\begin{equation*}
\delta+2^{1+n\alpha} R_0^{-n\alpha}\delta^{\alpha}t_0 < R_0/2
\end{equation*}
we define the  function $f^{\delta,m}:[m-1,m]\times [0,t_0] \rightarrow \mathbb{R}$ by
\begin{equation*}
f^{\delta, m}(h,t) \coloneqq R_0- \delta (h-m)^2 - 2^{1+n\alpha} R_0^{-n\alpha}\delta^{\alpha} t.
\end{equation*}
Let  $f^{-1}(\cdot,t)$ denote  the inverse function of $f(\cdot,t)$ and $\Phi^{\delta, m}_t$ denote the graph of the rotationally symmetric function $\varphi(y)=f^{-1}(|y|,t)$, namely 
\begin{equation*}
\Phi^{\delta, m}_t \coloneqq \{(y,h):|y| =f^{\delta, m}(h,t) \}  
\end{equation*}
(see in the figure above). 
By definition, we have
$
\Phi^{\delta, m}_0 \prec \Sigma_0.  
$
We  will show that $ \Phi^{\delta, m}_t $  defines a supersolution of \eqref{eq:INT aGCF} and consequently that  
\begin{equation*}
\Phi^{\delta, m}_{t_0} \prec \Sigma_{t_0}, \qquad \mbox{for} \,\, m > M_\delta+1
\end{equation*}
for some constant  $M_\delta$ depending on $\delta$. Thus,
\begin{equation*}
B_{f^{\delta,m}(m,t_0)}(0)=L_m(\Phi^{\delta, m}_{t_0}) \prec L_m(\Sigma_{t_0}) \preceq \Omega_{t_0}
\end{equation*}
and by passing $\delta$ to zero, we will obtain the desired result. 

Let us first see that $ \Phi^{\delta, m}_t $  defines supersolution of \eqref{eq:INT aGCF}. 
For convenience, we  denote $\bd_r (f^{-1})(r,t)$ and $\bd_{rr} (f^{-1})(r,t)$ by $f^{-1}_r$ and $f^{-1}_{rr}$, respectively. Then, the Gauss curvature $K$ of $\Phi^{\delta, m}_t$ satisfies
\begin{align*}
K = \frac{f^{-1}_{rr} |f^{-1}_r|^{n-1}}{r^{n-1}(1+|f^{-1}_r|^2)^{\frac{n+2}{2}}} \leq \frac{f^{-1}_{rr} }{|R_0/2|^{n-1}(1+|f^{-1}_r|^2)^{\frac{3}{2}}}=-\frac{2^{n-1}f_{hh} }{R_0^{n-1}(1+f_h^2)^{\frac{3}{2}}} \leq \frac{2^{n}\delta }{R_0^{n-1}}<2^nR_0^{-n}\delta.
\end{align*}
Also, the gradient function $\upsilon$ of $f^{-1}$ on $L_{[m-1,m)}(\Phi^{\delta, m}_t)$ satisfies
\begin{align*}
\upsilon = (1+|f^{-1}_r|^2)^{\frac{1}{2}}=(1+\frac{1}{4\delta^2(h-m)^2}  )^{\frac{1}{2}} \leq \frac{(1+4\delta^2(h-m)^2)^{\frac{1}{2}}}{2\delta(m-h)} \leq \frac{1}{\delta(m-h)}
\end{align*}
since $\delta <R_0/2<1/2$. Finally, on $L_{[m-1,m)}(\Phi^{\delta, m}_t)$, we can derive from $f^{-1}(f(h,t),t)=h$ the following
\begin{align*}
\bd_t (f^{-1}) = -f^{-1}_r \bd_t f =\frac{2^{1+n\alpha} R_0^{-n\alpha}\delta^{\alpha}}{2\delta(m-h)}= \frac{2^{n\alpha} R_0^{-n\alpha}\delta^{\alpha}}{\delta(m-h)}> K^{\alpha}\upsilon.
\end{align*}
Thus, $\Phi^{\delta, m}_t$ is a supersolution of \eqref{eq:INT aGCF}. In other words, $\Sigma_t$ cannot contact with $\Phi^{\delta, m}_t$ in the interior of $\Phi^{\delta, m}_t$. 

If there exists the contact time $t^{\delta,m}\in (0,t_0]$ such that $\Sigma_t \bigcap \cv(\Phi^{\delta, m}_t) =\emptyset$ for $t \in (0,t^{\delta,m})$ and $\Sigma_{t^{\delta,m}} 
\bigcap \cv(\Phi^{\delta, m}_{t^{\delta,m}}) \neq \emptyset$, then we denote the contact set by $Z_{\delta,m}$
\begin{equation*}
Z_{\delta,m} \coloneqq \Sigma_{t^{\delta,m}} \bigcap \cv(\Phi^{\delta, m}_{t^{\delta,m}})=\Sigma_{t^{\delta,m}} \bigcap \Phi^{\delta, m}_{t^{\delta,m}}.
\end{equation*}
We have just seen that 
\begin{equation*}
Z_{\delta,m} \subset \bd(\Phi^{\delta, m}_{t^{\delta,m}})=  L_{(m-1)}(\Phi^{\delta, m}_{t^{\delta,m}})\bigcup L_{m}(\Phi^{\delta, m}_{t^{\delta,m}})  
\end{equation*}
and  since $\Sigma_t$ is a graph, it  cannot contact $\Phi^{\delta, m}_{t}$ on $L_m(\Phi^{\delta, m}_t)$. Hence,
\begin{equation}\label{eq:LTE Contact set I}
Z_{\delta,m} \subset L_{(m-1)}(\Phi^{\delta, m}_{t^{\delta,m}}).  \tag{6.1}
\end{equation}
If there exist a contact point $(z_0,m-1) \in Z_{\delta,m}$, then $|z_0|=f^{\delta,m}(m-1,t_{\delta,m})< R_0$ and the slope of the graph of $u$ at this point is at most equal to  the slope of $\Phi^{\delta, m}$
(see in Figure 2(c)), hence 
$|Du|(z_0,t_{\delta,m})\leq (2\delta)^{-1}.$

On the other hand, since  $A \coloneqq \{ (y,t) :t\in [0,t_0] , y \in \Omega_t, |y|\leq R_0, |Du|(y,t) \leq (2\delta)^{-1} \}$ is a compact set, the function  $u$ attains its maximum in $A$. Set 
$$M \coloneqq \max\{u(y,t): (y,t)\in A \}.$$
Since $(z_0,t_{\delta,m})\in A$, we must have  $m-1=u(z_0,t_{\delta,m}) \leq M_\delta$. Therefore, if $m>M_\delta+1$, then
\begin{equation}\label{eq:LTE Contact set II}
Z_{\delta,m} \bigcap L_{(m-1)}(\Phi^{\delta, m}_{t^{\delta,m}}) =\emptyset  \tag{6.2}
\end{equation}
concluding that $\Phi^{\delta, m}_{t_0} \prec \Sigma_{t_0}$ and $B_{f^{\delta,m}(m,t_0)}(0)  \preceq \Omega_{t_0}$. 
By passing $\delta$ to zero, we obtain the desired result.

\end{proof}

\centerline{\bf Acknowledgements}

\smallskip 

\noindent The authors are indebted  to Pengfei  Guan for many fruitful discussions. 
\smallskip

\noindent P. Daskalopoulos  was  partially supported by NSF grant DMS-1266172.

\noindent Ki-Ahm Lee  was supported by the National Research Foundation of Korea (NRF) grant funded by the Korea government (MSIP) (No.2014R1A2A2A01004618). 
Ki-Ahm Lee also holds  a joint appointment with the Research Institute of Mathematics of Seoul National University.

\bibliographystyle{abbrv}

\bibliography{myref}

\end{document}